\documentclass[10pt]{amsart}

\usepackage{enumerate}
\usepackage[dvips]{color}
\usepackage{tikz-cd}
\usepackage{faktor}
\usepackage{mathtools}
\usepackage{amssymb}
\usepackage{epsfig}
\usepackage{amsmath}
\usepackage{color}
\usepackage{mathrsfs}
\usepackage{faktor}
\usepackage{xfrac}
\usepackage[linesnumbered,ruled]{algorithm2e}
\usepackage{url}
\usepackage{float}
\usepackage{authblk}
\newtheorem{theorem}{Theorem}
\newtheorem{lemma}[theorem]{Lemma}
\newtheorem{corollary}[theorem]{Corollary}
\newtheorem{proposition}[theorem]{Proposition}

\newtheorem{example}[theorem]{Example}

\title{Implicitization Of A Plane Curve Germ}

\begin{document}

\maketitle
\authors{Joao Cabral - {\small NOVA School of Science and Technology, Portugal; jpbc@fct.unl.pt} 
\and
Ana Casimiro  - {\small NOVA School of Science and Technology, Portugal; amc@fct.unl.pt} 
}


\begin{abstract}
Let $Y=\{f(x,y)=0\}$ be the germ of an irreducible plane curve. We present an algorithm to obtain polynomials, whose valuations coincide with the semigroup generators of $Y$. These polynomials are obtained sequentially, adding terms to the previous one in an appropriate way. To construct this algorithm, we perform truncations of the parametrization of $Y$ induced by the Puiseux Theorem. Then, an implicitization  theorem  of Tropical Geometry (Theorem $1.1$ of \cite{CM}) for plane curves is applied to the truncations. The identification between the local ring and the semigroup of $Y$ plays a key role in the construction of the algorithm, allowing us to carry out a formal elimination process, which we prove to be finite. The complexity of this elimination process equals the complexity of a integer linear programming problem. This algorithm also allows us to obtain an approximation of the series $f$, with the same multiplicity and characteristic exponents. We present a pseudocode of the algorithm and examples, where we compare computation times between an implementation of our algorithm and elimination through Gr\"{o}bner bases.
\end{abstract}

\section*{Acknowledgements}

\noindent
The authors were supported by national funds through the FCT - Fundação para a Ciência e a Tecnologia, I.P., under the scope of the projects UIDB/00297/2020 and UIDP/00297/2020 (Center for Mathematics and Applications).

\section{Introduction}

\noindent
The main goal of implicitization is to obtain a representation of an algebraic variety as the zero set of Laurent polynomials from a parametrization representation given also by Laurent polynomials. In \cite{STY},  B. Sturmfels characterizes implicitization as a fundamental operation in computational algebraic geometry and gives a general solution to the problem of implicitization in the hypersurface case. In this case, the implicitization procedure is usually done by using a theorem that characterizes a set, which contains the support of a Laurent polynomial defining the zero set, then it applies numerical linear algebra to compute the coefficients of the Laurent polynomial.

\noindent
In this paper, we undertake implicitization in a Local Analytic Geometry context and our approach will be different from the previous one. Puiseux Theorem produces an algorithm to obtain a parametrization of the germ of an irreducible plane curve (see Section 8.3 of \cite{BK}), such parametrization contains topological information and one can associate a semigroup to the irreducible plane curve using such information. We present a procedure that reduces implicitization to an integer linear programming problem using the irreducible plane curve semigroup. The following example illustrates the importance and difficulties of the implicitization problem. Let us consider the germ of a plane curve $Y$ with characteristic $(6;8,9)$. Then, by applying O. Zariski's techniques presented in \cite{ZAR}, one obtains that the semigroup of $Y$ is generated by $\{6,8,25\}$ and the curve admits a short form parametrization of the type
\begin{equation}\label{INTEX}
	x=t^6, y=t^8+ct^9+a_1 t^{10}+a_2 t^{11}+a_4 t^{13}+a_{10} t^{19},
\end{equation}
where $c\in\mathbb C^{\ast}$ and $a_i\in\mathbb C$, $i\in\{1,2,4,10\}$. By Theorem 2.64 of \cite{GLS}, this short parametrization induces a unique equisingular deformation of the equation of $Y$. One way to obtain the equation associated to (\ref{INTEX}) is to use a  computer algebra system for polynomial computations, for example \em Singular \em (see \cite{SING}), and apply elimination theory using Gr\"{o}bner basis. The problem lies in the fact that this procedure has very high memory usage and can be hard to get results on a standard pc. 

\noindent
In Section \ref{SQO} we introduce the basics about the germ of an irreducible quasi-ordinary hypersurface $Y$, namely characteristics exponents, semigroup and, in the case of plane curves, the connection between valuation and semigroup. In Theorem \ref{NORMALF} we present a substitute to the conductor of the semigroup for the hypersurface case. We also present some results that will be needed later on. 

\noindent
From Section \ref{NPE} onward we restrict ourselves to the case of plane curves. Following the reasoning of truncating a branch, with $s$ characteristic exponents, made in the proof of Theorem 3.9 of \cite{ZAR}, we introduce in Section \ref{NPE} truncations $\iota_i$, $i=1,\ldots,s$, of a parametrization and the induced valuations $\vartheta_{\iota_i}$. Since these truncations are polynomial, we apply a theorem of Tropical Geometry, Theorem $1.1$ of \cite{CM}, to obtain a polygon $N_i$ that contains the support of a polynomial $f_i$ such that $\iota_i^{\ast}f_i=0$. We relate polygon $N_i$ to the Newton polygon (in terms of Puiseux expansions) of $Y_i=\{f_i=0\}$. 

\noindent
In Section \ref{MAIN} we present the main result, Theorem \ref{MAINTHEO}, which states the existence of polynomials $f_i$, $i\in\{0,\ldots,s\}$, such that $f_{i-1}$ is obtained from $f_i$, the support of $f_i$ is contained in $N_i$ and $\iota_i^{\ast}f_i=0$. Furthermore, for $i=1,\ldots,s$, Theorem \ref{MAINTHEO} gives a relation between $\vartheta_{\iota_j}(f_{i-1})$, $j=i,\ldots,s$, and the generators of the semigroup of $Y_j$. These polynomials are unique up to multiplication by a constant. Corollary \ref{APPROX} states that $\{f_s=0\}$ is a "good" approximation of a representative of $Y$. The polynomial $f_s$ depends only of our choice for the truncation $\iota_s$. We prove Theorem \ref{MAINTHEO} by induction in Subsections \ref{MAINA} and \ref{MAINB}. In Subsection \ref{MAINA}, we prove the existence of the $f_i$'s, by doing a formal elimination procedure that gives us formal versions of the $f_i$'s and from these we prove the existence of the polynomials $f_i$'s. Subsection \ref{MAINB} is dedicated to the proof of the valuation part of the main result. Lemma \ref{INDUI} gives us relations between the valuations $\vartheta_{\iota_i}$. One important fact to obtain these relations is the decomposition given in statement \em (B) \em of Lemma \ref{INDUI} and also present in the proof of Theorem 3.9 of \cite{ZAR}. The valuations $\vartheta_{\iota_i}$ are induced by polynomial parametrizations and therefore can be implemented on a computer.

\noindent
Based upon step 2 made in the proof of  statement \em (B) \em of Theorem \ref{MAINTHEO}, we give in Section \ref{CA} an algorithm to compute $f_i$ from $f_{i-1}$. Whereas the elimination procedure used in step 2 was formal, now using the appropriate characterization of the polygon's $N_i$ and the decomposition given in statement \em (B) \em of Lemma \ref{INDUI}, Algorithm \ref{COMPIMP} gives us a finite elimination procedure whose computational complexity depends only on a integer linear programming problem. Namely, finding the points with non negative integers entries on a compact section of a hyperplane, whose normal vector is given by the generators of the semigroup associated to $Y_i$. Algorithm \ref{COMPIMP} also depends on computing valuations, but these computations are done only using the $\vartheta_{\iota_i}$'s. To finalize we give two examples comparing computing times between an implementation of Algorithm \ref{COMPIMP} and elimination using \em Singular\em.

\noindent
The authors intend in the future to generalize the results of this paper to Quasi-Ordinary hypersurfaces.

\section{Semi-group of a Quasi-Ordinary Hypersurface}\label{SQO}

\noindent
Let $\mathbb N_0$ [$\mathbb N$] be the set of non negative integers [positive integers] and $\mathbb Q_{\geq 0}$ [$\mathbb Q_{> 0}$] the set of non negative rationals [positive rationals]. We endow $\mathbb R^n$ with the partial order
\begin{equation}\label{PARTORD}
	(a_1,\ldots,a_n)\leq (b_1,\ldots,b_n) \Leftrightarrow a_i\leq b_i,\,\forall i\in\{1,\ldots,n\}.
\end{equation}
If $(a_1,\ldots,a_n)\leq (b_1,\ldots,b_n)$ and $(a_1,\ldots,a_n)\neq (b_1,\ldots,b_n)$, we write $(a_1,\ldots,a_n)<(b_1,\ldots,b_n)$. Note that for $n=1$ we have the usual order in $\mathbb R$.

\noindent
Set $x=(x_1,\ldots,x_n)$. For $\alpha=(\alpha_1,\ldots,\alpha_n)\in \mathbb Q_{\geq 0}^n$, set  $|\alpha|=\sum_{i=1}^n\alpha_i$ and
\begin{equation}\label{VECTPOWER}
	x^{\alpha}=x_1^{\alpha_1}\cdots x_n^{\alpha_n}.
\end{equation}
Let $f(x)=\sum_{|\alpha|\geq 0}a_{\alpha}x^{\alpha}\in\mathbb C[[x_1,\ldots,x_n]]$. We define the \em support of \em $f$, and represent it by $\textrm{Supp}(f)$, as the set
\[
\textrm{Supp}(f)=\{\alpha\in\mathbb N_0^n:a_{\alpha}\neq 0\}.
\]
By definition, the support of the null serie is the empty set. Write $f(x)=\sum_{i\geq m}f_i(x)$, where $f_i$ is a homogenous polynomial of degree $i$ for all $i\geq m$ and $f_m$ is a non null polynomial. We define $\textrm{ord}(f)$, the \em order \em of $f$, as $m$. If $f\in\mathbb C\{x\}$ define $C_0(\{f(x)=0\})$, the \em tangent cone \em of $\{f(x)=0\}$, as the algebraic variety $\{f_m(x)=0\}$. For a polynomial $g\in\mathbb C[x]$, denote the degree of $g$ by $\textrm{deg}(g)$.

\noindent
Set $M$ as a complex manifold of dimension $n+1$, $o\in M$ and $Y$ an irreducible hypersurface of $M$. The germ of $Y$ at $o$ is \em quasi-ordinary \em if there is a system of local coordinates $(x,y)$ centered at $o$ such that the discriminant of $Y$ with respect to the map $p(x,y)=x$ is contained in $\{x_1\cdots x_n=0\}$. From now on $Y=\{f(x,y)=0\}$, $f\in\mathbb C\{x,y\}$, will denote the germ of an irreducible quasi-ordinary hypersurface at the origin. The roots of $f$ are denominated \em branches \em of $Y$ and are fractional power series in the ring $\mathbb C\{x_1^{1/k},\ldots,x_n^{1/k}\}$ (we will abbreviate this notation by $\mathbb C\{x^{1/k}\}$), for some positive integer $k$ (see \cite{LIP}). Let $\zeta(x^{1/k})\in \mathbb C\{x^{1/k}\}$ be a branch of $Y$. There are $\lambda_1,\ldots,\lambda_s\in\mathbb Q^n_{\geq 0}$ such that
\begin{enumerate}
	\item[\em C1)\em]
	$\lambda_1<\lambda_2<\ldots<\lambda_s$;
	\item[\em C2)\em]
	Set $M_0=\mathbb Z^n$ and $M_i=M_0+\mathbb Z^n\lambda_i$, $i=1,\ldots,s$. For all $i\in\{1,\ldots,s\}$, $\lambda_i\not \in M_{i-1}$;
	\item[\em C3)\em]
	$\textrm{ord}(f)=k=k_1k_2\ldots k_s$, where $k_i$ is the smallest positive integer such that $k_i\lambda_i\in M_{i-1}$;
	\item[\em C4)\em]
	For all $i\in\{1,\ldots,s\}$, the coefficient of $x^{\lambda_i}$ in $\zeta(x^{1/k})$ is non null.
	\item[\em C5)\em]
	Let $\alpha\in \mathbb Q^n_{>0}\setminus \{\lambda_1,\ldots,\lambda_s\}$. If  the coefficient of $x^{\alpha}$ in $\zeta(x^{1/k})$ is non null then there is $i\in\{1,\ldots,s\}$ such that $\alpha\in M_i$. Furthermore, if $i$ is the smallest element of $\{1,\ldots,s\}$ such that $\alpha\in M_i$ then $\lambda_i<\alpha$.
\end{enumerate}
If $\xi(x^{1/k})$ is another branch of $Y$ then $\xi(x^{1/k})=\zeta(\theta x^{1/k})$, where $\theta$ is the $k$-th root of unity. We obtain all branches of $Y$ by running through all $k$-th roots of unity. The rational $n$-uples $\lambda_1,\ldots,\lambda_s$ are denominated \em characteristic exponents of \em $Y$ (see Section 1 of \cite{LIP}). Furthermore, if $\zeta\in\mathbb C\{x^{1/k}\}$ admits $\lambda_1,\ldots,\lambda_s\in\mathbb Q^n_{\geq 0}$ that verify $C1)$, $C2)$, $C4)$ and  $C5)$ then $\zeta$ is the branch of a quasi-ordinary hypersurface (see Proposition 1.5 of \cite{LIP}). After a change of coordinates of the type
\begin{equation}\label{CLEANUP}
(x,y)\to (x,y-\varphi(x)),\, \varphi\in\mathbb C\{x\},
\end{equation}
we can assume that no branch of $Y$ contains monomials with positive integer exponent and non null coefficient.

\noindent
By default, zero will belong to all semigroups mentioned on this paper. Set $e_0=1$ and $e_i=k_ie_{i-1}$, $i=1,\ldots,s$. Note that $e_s=k$. Let $u_1,\ldots,u_n$ be the canonical basis of $\mathbb R^n$. For $i\in\{1,\ldots,s\}$, set $\Gamma(\lambda_1,\ldots,\lambda_i)$ as the semigroup generated by $e_iu_1,\ldots, e_iu_n$,
\begin{equation}\label{GEN}
\gamma_1^{(i)}=e_i\lambda_1 \textrm{ and } \gamma_{j+1}^{(i)}=k_j\gamma_j^{(i)}-e_i\lambda_j+e_j\lambda_{j+1}\textrm{ for } j=1,\ldots, i-1.
\end{equation}
The set $\Gamma(\lambda_1,\ldots,\lambda_i)$  is the semigroup of a quasi-ordinary surface with characteristic exponents $\lambda_1,\ldots,\lambda_i$. We will also denote the generators $\gamma_1^{(s)},\ldots,\gamma_s^{(s)}$ of $\Gamma(\lambda_1,\ldots,\lambda_s)$, the semigroup of $Y$, by, respectively,  $\gamma_1,\ldots,\gamma_s$. For $i=1,\ldots,s$, set $\Gamma_i(\lambda_1,\ldots,\lambda_s)$ as the subsemigroup of $\Gamma(\lambda_1,\ldots,\lambda_s)$ generated by $ku_1,\ldots, ku_n$ and $\gamma_j^{(s)}$, $j=1,\ldots,i$. Next Proposition gives us a relation between  $\Gamma_i(\lambda_1,\ldots,\lambda_s)$ and  $\Gamma(\lambda_1,\ldots,\lambda_i)$.

\begin{proposition}\label{SEMIREL}
For all $i=1,\ldots,s-1$,
\begin{equation}\label{CONVERT}
\Gamma_i(\lambda_1,\ldots,\lambda_s)=\frac{e_s}{e_i} \Gamma(\lambda_1,\ldots,\lambda_i).
\end{equation}
\end{proposition}

\begin{proof}
We have immediately the following relations
\[
e_su_j=\frac{e_s}{e_i} (e_i u_j),\, j=0,1,\ldots,n,
\]
\[
\gamma_1^{(s)}=\frac{e_s}{e_i}(e_i\lambda_1)=\frac{e_s}{e_i}\gamma_1^{(i)},
\]
\begin{align*}
\gamma_{j+1}^{(s)}=& k_j\gamma_j^{(s)}-e_s\lambda_j+e_s\lambda_{j+1}=
k_j\frac{e_s}{e_i} \gamma^{(i)}_{j}-\frac{e_s}{e_i} e_i\lambda_j+\frac{e_s}{e_i} e_i\lambda_{j+1}=\\
=& \frac{e_s}{e_i}\gamma^{(i)}_{j+1},\, j=1\ldots,i-1.
\end{align*}
\end{proof}

\noindent
On Theorem \ref{NORMALF}, we prove the following: Let $\widetilde{\Gamma}$ be the subgroup of $\mathbb Z^n$ generated by $\Gamma(\lambda_1,\ldots,\lambda_s)$, it exists $c\in \mathbb N_0^n$ such that 
\begin{equation}\label{QUASIC}
	\forall a\in\widetilde{\Gamma}, c\leq a \Rightarrow a\in \Gamma(\lambda_1,\ldots,\lambda_s).
\end{equation}
This results gives us a substitute for the conductor of the semigroup of a plane curve.  The following Lemma is needed for the proof of Theorem \ref{NORMALF}.

\vspace{0.4cm}

\begin{lemma}\label{ALTGENRETI}
The following statements hold.
\begin{enumerate}
\item[$(A)$]
For $i=0,\ldots,s-1$,
\begin{equation}\label{GENALT}
\gamma_{i+1}=k\lambda_{i+1}+\sum_{j=1}^{i}(k_j-1)\gamma_j.
\end{equation}
\item[$(B)$]
For all $i=1,\ldots,s$, $kM_i= k\mathbb Z^n+\sum_{j=1}^{i}\mathbb Z \gamma_j$.
\end{enumerate}
\end{lemma}

\begin{proof}
We will prove both statements by induction on $i$. We begin with statement \em (A) \em. By (\ref{GEN}),
\begin{align*}
\gamma_{i+1}=&k_i\gamma_i-k\lambda_i+k\lambda_{i+1}=(k_i-1)\gamma_i+\gamma_i-k\lambda_i+k\lambda_{i+1}=\\
=&(k_i-1)\gamma_i+k\lambda_{i}+\sum_{j=1}^{i-1}(k_j-1)\gamma_j-k\lambda_i+k\lambda_{i+1}=k\lambda_{i+1}+\sum_{j=1}^{i}(k_j-1)\gamma_j.
\end{align*}
We now proceed to prove statement \em (B) \em. Let $a\in k\mathbb Z^n+\sum_{j=1}^{i+1}\mathbb Z \gamma_j$.
There exist integers $\alpha_j$, $j=1,\ldots,n$, and $\beta_j$, $j=1,\ldots,i+1$ such that
\begin{equation}\label{COEFCOMB}
	a=\sum_{j=1}^{n}k\alpha_j u_j+\sum_{j=1}^{i+1}\beta_j\gamma_j.
\end{equation}
Set $b=\sum_{j=1}^{n}k\alpha_j u_j+\sum_{j=1}^{i}\beta_j\gamma_j$. Applying equality (\ref{GEN}), we obtain
\[
a=b+\beta_{i+1}k_i\gamma_i-\beta_{i+1}k\lambda_i+\beta_{i+1}k\lambda_{i+1}.
\]
By the induction hypothesis, since $b+\beta_{i+1}k_i\gamma_i\in k\mathbb Z^n+\sum_{j=1}^{i}\mathbb Z \gamma_j$, there is $\delta \in kM_i$ such that
\[
a=\delta-\beta_{i+1}k\lambda_i+\beta_{i+1}k\lambda_{i+1}.
\]
Hence $a\in kM_{i+1}$. 

\noindent
Assume now that $a\in kM_{i+1}$. Taking into account that from (\ref{GEN}) we obtain the equality  $k\lambda_{j+1}=k\lambda_j-k_j\gamma_j+\gamma_{j+1}$, the proof that $a\in k\mathbb Z^n+\sum_{j=1}^{i+1}\mathbb Z \gamma_j$ is similar to the previous one.
\end{proof}

\begin{theorem}\label{NORMALF}
Let $i=1,\ldots,s$ and $a\in k\mathbb Z^n+\sum_{j=1}^{i}\mathbb Z \gamma_j$. There are integers $\alpha_j$, $j=1,\ldots,n$, and $\beta_j$, $j=1,\ldots,i$, such that $0\leq \beta_j\leq k_j-1$ and  $a=\sum_{j=1}^{n}\alpha_j ku_j+\sum_{j=1}^{i}\beta_j\gamma_j$. Furthermore, if $a\geq \sum_{j=1}^{i}(k_j-1)\gamma_j$ then $\alpha_j\geq 0$ for all $j=1,\ldots,n$. 
\end{theorem}

\begin{proof}
We will proof the result by induction on $i$. There are integers $\alpha'_j$, $j=1,\ldots,n$, and $\beta'_j$, $j=1,\ldots,i$ such that
\begin{equation}\label{SEMICOMB}
	a=\sum_{j=1}^{n}\alpha'_jku_j+\sum_{j=1}^{i}\beta'_j\gamma_j.
\end{equation}
Let $\ell$ be the  integer such that $0\leq \beta_i-\ell k_i <k_i$. We can rewrite (\ref{SEMICOMB}) as
\[
a=\sum_{j=1}^{n}\alpha'_jku_j+\sum_{j=1}^{i-1}\beta'_j\gamma_j+(\beta'_i-\ell k_i)\gamma_i+\ell k_i\gamma_i.
\]
By (\ref{GENALT}),
\[
k_i\gamma_i=k_i k\lambda_i+\sum_{j=1}^{i-1}k_i(k_j-1)\gamma_j.
\]
Since $k_i\lambda_i\in M_{i-1}$, by statement \em (B) \em of Lemma \ref{ALTGENRETI}, $k(k_i\lambda_i)\in k\mathbb Z^n+\sum_{j=1}^{i-1}\mathbb Z \gamma_j$. Using the induction hypothesis, there are $\alpha_j$, $j=1,\ldots,n$, and $\beta_j$, $j=1,\ldots,i-1$, such that $0\leq \beta_j\leq k_j-1$ and 
\[
\sum_{j=1}^{n}\alpha'_jku_j+\sum_{j=1}^{i-1}\beta'_j\gamma_j+\ell k_i\gamma_i=\sum_{j=1}^{n}\alpha_jku_j+\sum_{j=1}^{i-1}\beta_j\gamma_j.
\]
Set $\beta_i=\beta'_i-\ell k_i$. Assume that
\[
a=\sum_{j=1}^{n}\alpha_jku_j+\sum_{j=1}^{i}\beta_j\gamma_j\geq \sum_{j=1}^{i}(k_j-1)\gamma_j.
\]
Hence
\[
\sum_{j=1}^{n}\alpha_jku_j\geq \sum_{j=1}^{i}(k_j-(\beta_j+1))\gamma_j.
\]
Since $0\leq \beta_j\leq k_j-1$ and $k_j\geq 2$ for all $j=1,\ldots,i$, we conclude that $\alpha_j\geq 0$ for all $j=1,\ldots,n$.
\end{proof}

\noindent
In the following proposition we prove some useful properties. Statement \em (C) \em will play an important role in the proof of statement $(A)$ of Theorem \ref{MAINTHEO}. Statement \em (D) \em will be used in the proof of Lemma \ref{LOWERMULT}.

\begin{proposition}\label{EQINEQ}
For all $i=1,\ldots,s$, the following statements hold:
\begin{enumerate}
	\item[$(A)$]
	$\displaystyle \sum_{\ell=1}^{i}(k_{\ell}-1)\gamma_{\ell}=k_i\gamma_i-k\lambda_i$;
	\item[$(B)$]
	$\displaystyle k_i\gamma_i>\sum_{\ell=1}^{i}(k_{\ell}-1)\gamma_{\ell}$;
	\item[$(C)$]
	$k_i\gamma_i\in \Gamma_{i-1}(\lambda_1,\ldots,\lambda_s)$, for all $i=1,\ldots,s$;
	\item[$(D)$]
	 For $s\geq 2$ and $i=1,\ldots,s$, 	
	\[
	e_{i-1}\gamma_{s}-e_{s-1}\gamma_{i}>0.
	\]
\end{enumerate}
\end{proposition}

\begin{proof}
By statement \em (A) \em of Lemma \ref{ALTGENRETI},
\[
\displaystyle \sum_{\ell=1}^{i}(k_{\ell}-1)\gamma_{\ell}=(k_i-1)\gamma_i+\displaystyle \sum_{\ell=1}^{i-1}(k_{\ell}-1)\gamma_{\ell}=(k_i-1)\gamma_i+ \gamma_i-k\lambda_i=k_i\gamma_i-k\lambda_i.
\]
Statement $(B)$ results immediately from statement $(A)$.

\noindent
We prove $(C)$ by induction on $i$. By (\ref{GEN}), the equality
\[
k_{i+1}\gamma_{i+1}=k_{i+1}(k_i\gamma_i)-k_{i+1}(k\lambda_i)+k(k_{i+1}\lambda_{i+1})
\]
holds. From the induction hypothesis, $k_i\gamma_i$ belongs to $\Gamma_{i-1}(\lambda_1,\ldots,\lambda_s)$, which is a subsemigroup of $\Gamma_{i}(\lambda_1,\ldots,\lambda_s)$.
Given that $k\lambda_i\in kM_i$, statement \em (B) \em of Lemma \ref{ALTGENRETI} implies that  $k\lambda_i\in k\mathbb Z^n+\sum_{j=1}^{i}\mathbb Z \gamma_j$.
Since $k_{i+1}\lambda_{i+1}\in M_{i}$, then once again by statement \em (B) \em of Lemma \ref{ALTGENRETI}, we conclude that $k(k_{i+1}\lambda_{i+1})\in k\mathbb Z^n+\sum_{j=1}^{i}\mathbb Z \gamma_j$. Hence $k_{i+1}\gamma_{i+1}\in k\mathbb Z^n+\sum_{j=1}^{i}\mathbb Z \gamma_{j}$. By statement $(B)$, 
\[
k_{i+1}\gamma_{i+1}>\sum_{\ell=1}^{i+1}(k_{\ell}-1)\gamma_{\ell}.
\]
Furthermore, the inequality 
\[
\sum_{j=1}^{i+1}(k_j-1)\gamma_{j}\geq \sum_{j=1}^{i}(k_j-1)\gamma_{j}
\]
holds trivially. Therefore, Theorem \ref{NORMALF} allows us to conclude that  $k_{i+1}\gamma_{i+1}$ belongs to the semigroup $\Gamma_{i}(\lambda_1,\ldots,\lambda_s)$.

\noindent
To prove statement $(D)$ we rewrite the quantity  $e_{i-1}\gamma_{s}-e_{s-1}\gamma_{i}$ in the following manner:
\begin{multline*}
	e_{i-1}\gamma_{s}-e_{s-1}\gamma_{i}=e_{i-1}\gamma_{s}-e_{s-1}\gamma_{i}+\sum_{j=i+1}^{s-1}\frac{e_{i-1}}{e_{j-1}}e_{s-1}\gamma_j-\sum_{j=i+1}^{s-1}\frac{e_{i-1}}{e_{j-1}}e_{s-1}\gamma_j=\\
	=\sum_{j=i+1}^{s}\frac{e_{i-1}}{e_{j-1}}e_{s-1}\gamma_j-\sum_{j=i}^{s-1}\frac{e_{i-1}}{e_{j-1}}e_{s-1}\gamma_j=\sum_{j=i+1}^{s}\frac{e_{i-1}}{e_{j-1}}e_{s-1}\gamma_j-\sum_{j=i+1}^{s}\frac{e_{i-1}}{e_{j-2}}e_{s-1}\gamma_{j-1}=\\
	=\sum_{j=i+1}^{s}\frac{e_{i-1}}{e_{j-1}}e_{s-1}\left(\gamma_j-\frac{e_{j-1}}{e_{j-2}}\gamma_{j-1}\right)=	\sum_{j=i+1}^{s}\frac{e_{i-1}}{e_{j-1}}e_{s-1}\left(\gamma_j-k_{j-1}\gamma_{j-1}\right).
	\qquad\qquad
\end{multline*}
For all $i=1,\ldots,s$, by definition $k_i>0$ which implies $e_i>0$. Furthermore, from  (\ref{GEN}) we conclude that $\gamma_i-k_{i-1}\gamma_{i-1}>0$. Therefore
\[
\sum_{j=i+1}^{s}\frac{e_{i-1}}{e_{j-1}}e_{s-1}\left(\gamma_j-k_{j-1}\gamma_{j-1}\right)>0.
\]
\end{proof}

\noindent
For the rest of this paper we will assume that $n=1$, so $Y$ is the germ at the origin of an irreducible plane curve.  The change of coordinates (\ref{CLEANUP}) allows us to assume that $C_0(Y)=\{y=0\}$. The characteristic exponents of an irreducible plane curve are equisingularity invariants (see Theorem 21 of \cite{BK}). For the \em Puiseux characteristic \em version of the characteristic exponents see Section 3.1 of \cite{WALL} and Section 3 of \cite{ZAR}. 

\noindent
To a branch $\zeta(x^{1/k})$ of $Y$, we associate a parametrization $\iota$ of $Y$ defined by $x=t^k,\,y=\zeta(t)$. By \em C4) \em and \em C5)\em, we can write the parametrization $\iota$ as
\begin{equation}\label{PAR}
	x=t^k,\, y=c_1t^{k\lambda_1}+\varphi_1(t)+c_2t^{k\lambda_2}+\varphi_2(t)+\cdots+c_st^{k\lambda_s}+\psi,
\end{equation}
where 
\begin{enumerate}
	\item[$(p1).$]
	$c_1,\ldots,c_s\in\mathbb C^{\ast}$, $\varphi_1,\ldots,\varphi_{s-1}\in(t)\mathbb C[t]$ and $\psi\in(t)\mathbb C\{t\}$;
	\item[$(p2).$]
	for all $i\in\{1,\ldots,s-1\}$, $\textrm{ord}(\varphi_i)>k\lambda_i$ and $\textrm{Supp}(\varphi_i)\subset kM_i$. Furthermore $\textrm{ord}(\psi)>k\lambda_s$ and $\textrm{Supp}(\psi)\subset kM_s$;
	\item[$(p3).$]
	for all $i\in\{1,\ldots,s-1\}$, $\textrm{deg}(\varphi_i)<k\lambda_{i+1}$ .
\end{enumerate}
The parametrization $\iota$ induces a map $\vartheta_{\iota}:\mathbb C[[x,y]]\to \mathbb N_0^n\cup \{+\infty\}$ defined by
\[
\vartheta_{\iota}(f)=\textrm{ord}(\iota^{\ast}f).
\]
By convention, $\vartheta_{\iota}(0)=+\infty$. This map verifies
\begin{eqnarray}
	\label{VALPROD}
	\vartheta_{\iota}(fg)=\vartheta_{\iota}(f)+\vartheta_{\iota}(g),\\
	\label{VALSUM}
	\vartheta_{\iota}(f+g)\geq \min\{\vartheta_{\iota}(f),\vartheta_{\iota}(g)\},
\end{eqnarray}
that is, $\vartheta_{\iota}$ is a valuation. Zariski proved in Theorem 3.9 of \cite{ZAR} that 
\begin{equation}\label{VALSEMIG}
\vartheta_{\iota}(\mathbb C[[x,y]] )=\Gamma(\lambda_1,\ldots,\lambda_s)\cup \{+\infty\}.
\end{equation}
This equality will play an important role in Section \ref{MAIN}.

\section{Newton Polytope of the Equation}\label{NPE}

\noindent
Let $\iota$ be as in (\ref{PAR}). Set $e_0=1$ and $e_i=k_ie_{i-1}$, $i=1,\ldots,s$ (see Section \ref{SQO}). For $i=1,\ldots,s$, define $\iota_i$ as the polynomial parametrization
\begin{equation}\label{PARi}
	x=t^{e_i},\, y=c_1t^{e_i\lambda_1}+\varphi_1(t^{e_i/k})+\cdots+c_it^{e_i\lambda_i}+\varphi_i(t^{e_i/k}),
\end{equation}
where $\textrm{Supp}(\varphi_s)\subseteq \textrm{Supp}(\psi)$. The parametrization $\iota_i$, $i\in\{1,\ldots,s\}$, is associated to a irreducible plane curve $Y_i$ (see Proposition 1.5. of \cite{LIP}) with characteristic exponents $\lambda_1,\ldots,\lambda_i$, multiplicity $e_i$ and branch
\begin{equation}\label{BRANCHi}
\zeta(x^{1/e_i})=c_1x^{\lambda_1}+\varphi_1(x^{1/e_i})+\cdots+c_{i}x^{\lambda_{i}}+\varphi_{i}(x^{1/e_i}).
\end{equation}
Furthermore, by our choice of coordinates, $C_0(Y_i)=\{y=0\}$. If one is interested in studying only the topology of $Y$ then one can consider $\varphi_s$ the null series.

\noindent
By Theorem $1.1$ of \cite{CM}, for $i\in\{1,\ldots,s\}$, there is $f_i\in (x,y)\mathbb C[x,y]$ such that $\iota_i^{\ast}f_i=0$, that is, $Y_i=\{f_i(x,y)=0\}$. Furthermore, Theorem $1.1$ of \cite{CM} allows us to compute a set $N_i$ such that $\textrm{Supp}(f_i)\subseteq N_i$. 

\noindent
For $h\in\mathbb C\{t\}$ and $a\in \mathbb C$, define $m_{a}(h)$ as the multiplicity of $a$ as zero of $h$ and $m_{\infty}(h)=-\textrm{deg}(h)$. For $i=1,\ldots,s$, if $\varphi_i$ is non null, define
\[
\mu_i=\frac{1}{e_i}\textrm{deg}(\varphi_i(t^{e_i/k}))\in M_i,
 \]
otherwise define $\mu_i=\lambda_i$. 

\noindent
Fix $i\in\{1,\ldots,s\}$. Let $a_j$, $j=1,\ldots,\tau_i$, be the zeros of $t^{-e_i\lambda_1}\iota_i^{\ast}y$, which is a polynomial with non null independent term $c_1$. Set $\kappa_j=m_{a_j}(t^{-e_i\lambda_i}\iota_i^{\ast}y)\neq 0$. Let
\[
\vec{v}_0=(m_0(\iota_i^{\ast}x),m_0(\iota_i^{\ast}y))=(e_i,e_i\lambda_1),
\]
\[
\vec{v}_{a_j}=(m_{a_j}(\iota_i^{\ast}x),m_{a_j}(\iota_i^{\ast}y))=(0,\kappa_i),\, j=1,\ldots,\tau_i,
\]
and
\[
\vec{v}_{\infty}=(m_{\infty}(\iota_i^{\ast}x),m_{\infty}(\iota_i^{\ast}y))=(-e_i,-e_i\mu_i).
\]
By construction
\[
\sum_{j=1}^{\tau_i}\vec{v}_{a_j}=(0,e_i\mu_i-e_i\lambda_1).
\]
As described in Section 1 of \cite{CM}, we rotate the vectors $\vec{v}_0$, $\vec{v}_{\infty}$ and $\vec{v}_{a_j}$, $j=1,\ldots,\tau_i$ by 90 degrees clockwise and concatenating them following their directions counter-clockwise, obtaining a polytope $B_i$. To $B_i$   we apply a translation such that the new polytope $B_i$ is contained in the first quadrant and intersects the axis. Hence $B_i$ is the convex hull of $\{(0,e_i),(e_i\lambda_1,0),(e_i\mu_i,0)\}$. Set
\begin{multline}\label{Ni}
N_i=B_i\cap \mathbb Z^2=\\
\{(\alpha,\beta)\in\mathbb N_0^2:k_1\alpha+(k_1\lambda_1)\beta\geq e_i(k_1\lambda_1)\wedge e_i\alpha+(e_i\mu_i)\beta\leq e_i(e_i\mu_i)\}.
\end{multline}
We give another useful characterization of the set $N_i$. Let $(\alpha,\beta)\in\mathbb N_0^2\setminus\{(0,0)\}$. we have
\[
\textrm{Supp}(\iota_i^{\ast}x^{\alpha}y^{\beta})=\left\{e_i\alpha+\sum_{\ell=1}^{\beta}a_{\ell}:a_1,\ldots,a_{\beta}\in\textrm{Supp}(\iota_i^{\ast}y)\right\}.
\]
Hence
\[
\min \textrm{Supp}(\iota_i^{\ast}x^{\alpha}y^{\beta})=e_i\alpha+(e_i\lambda_1)\beta,\,\max \textrm{Supp}(\iota_i^{\ast}x^{\alpha}y^{\beta})=e_i\alpha+(e_i\mu_i)\beta.
\]
Therefore $(\alpha,\beta)\in N_i$ if and only if
\begin{equation}\label{NiALT}
	\textrm{Supp}(\iota_i^{\ast}x^{\alpha}y^{\beta})\subset [e_i(e_i\lambda_1),e_i(e_i\mu_i)]
\end{equation}

\begin{example}\label{EX1.1}
Let $c_{\ell}\in\mathbb C^{\ast}$, $\ell=1,2,3$, and $\zeta=c_1 x^{3/2}+c_2 x^{5/3}+c_3 x^{23/12}$. Then $\zeta$ is a branch of an irreducible plane curve $Y$ with characteristic exponents $\lambda_1=3/2$,  $\lambda_2=5/3$ and  $\lambda_3=23/12$. We have $k_1=2$ and $k_2=3$. Since
\[
2\lambda_3=2\lambda_2+\lambda_1-1\in M_2
\]
then $k_3=2$. The semigroup $\Gamma(\lambda_1,\lambda_2)$ is generated by 
$\{6,9,19\}$ and $\Gamma(\lambda_1,\lambda_2,\lambda_3)$ is generated by 
$\{12,18,38,117\}$. We present $N_2$ and $N_3$ in figures (\ref{exem1N2}) and (\ref{exem1N3}), respectively.
\begin{figure}[h!]
	\centering
	\includegraphics[scale=0.4]{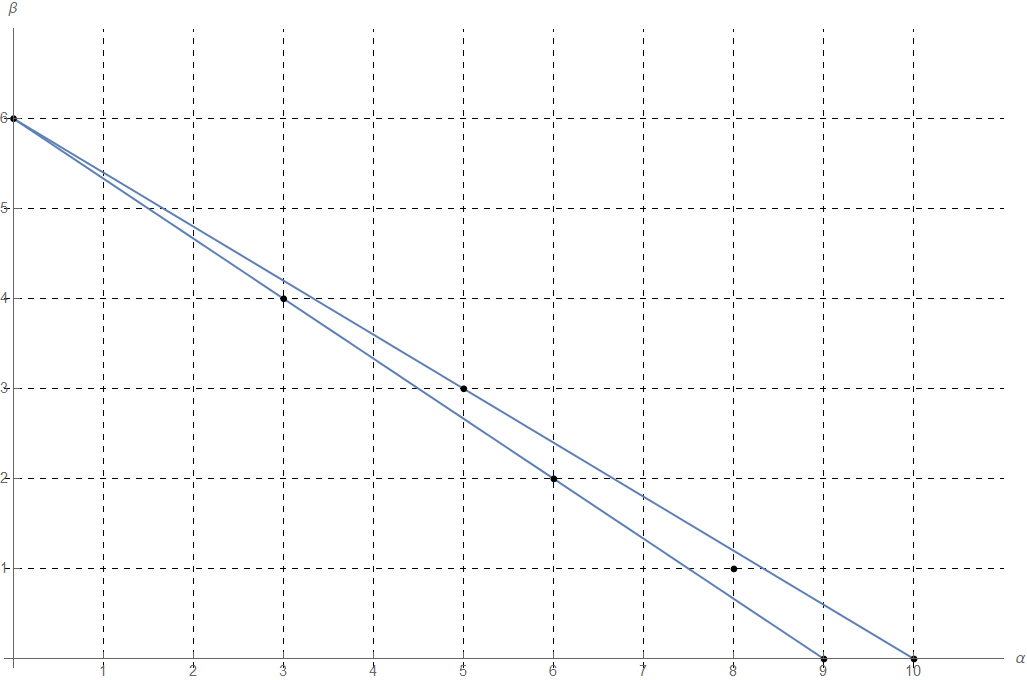}
	\caption{$N_2$}
	\label{exem1N2}
\end{figure}
\begin{figure}[h!]
	\centering
	\includegraphics[scale=0.4]{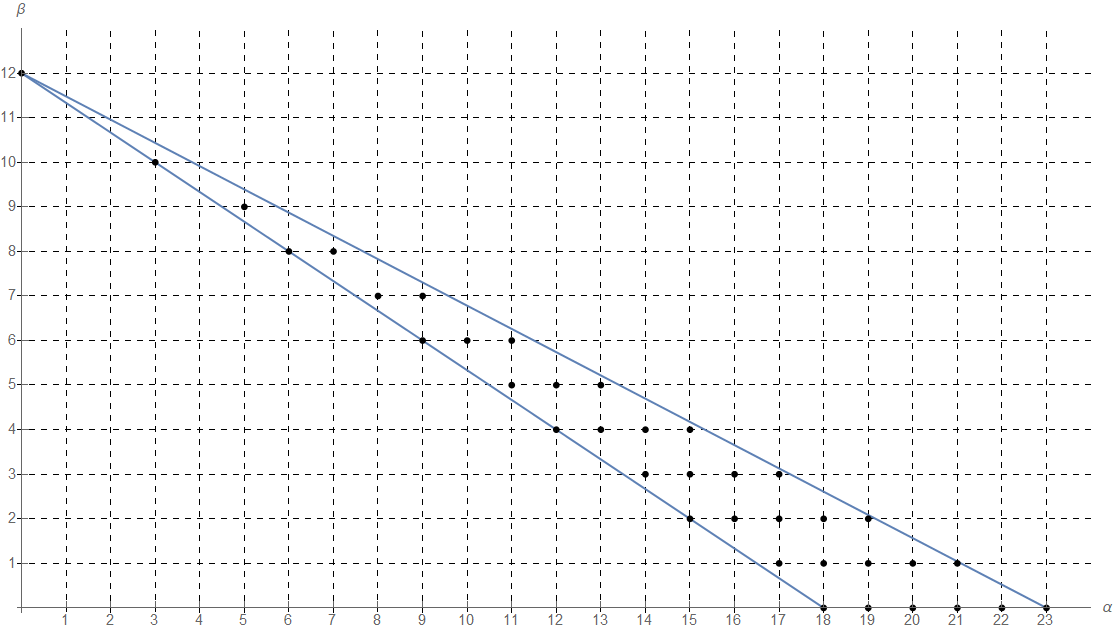}
	\caption{$N_3$}
	\label{exem1N3}
\end{figure}
\end{example}

\begin{example}\label{EX2.1}
Let $c_{\ell}\in\mathbb C^{\ast}$, $\ell=1,2,3$, and $\zeta=c_1 x^{3/2}+c_2 x^{5/2}+c_3 x^{8/3}+c_4 x^{10/3}$. We have
\[
\frac{5}{2}=\frac{3}{2}+1 \textrm{ and } \frac{10}{3}=2\,\frac{8}{3}-2.
\]
Then $\zeta$ is a branch of an irreducible plane curve $Y$ with characteristic exponents $\lambda_1=3/2$ and  $\lambda_2=8/3$. Furthermore, $k_1=2$ and $k_2=3$. The semigroup $\Gamma(\lambda_1)$ is generated by 
$\{2,3\}$ and $\Gamma(\lambda_1,\lambda_2)$ is generated by 
$\{6,9,25\}$. We present $N_1$ and $N_2$ in figures (\ref{exem2N1}) and (\ref{exem2N2}), respectively.
\begin{figure}[h!]
	\centering
	\includegraphics[scale=0.4]{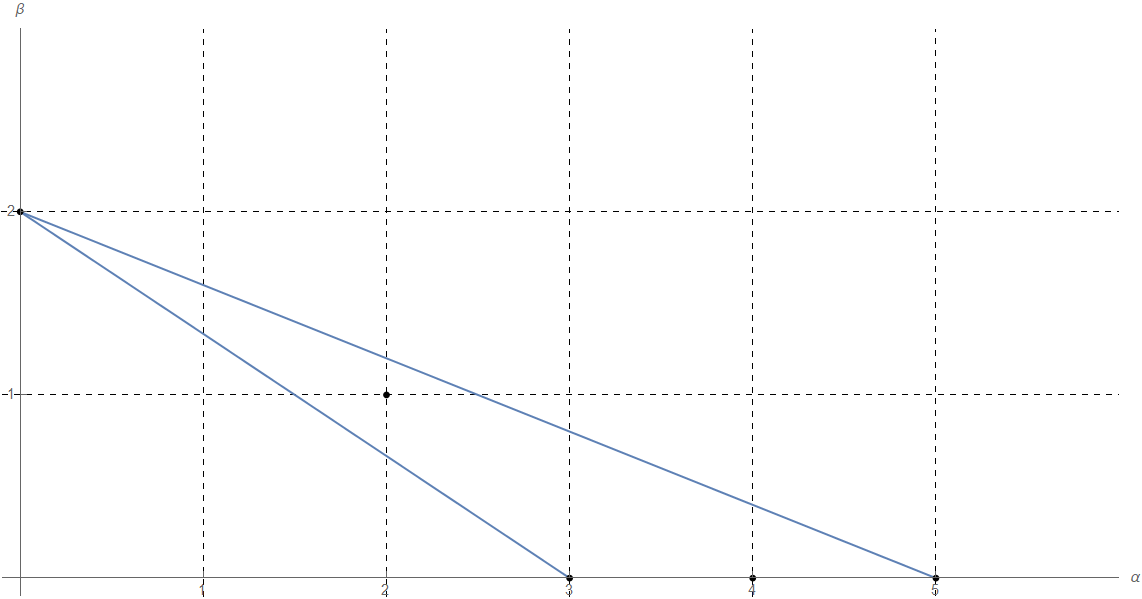}
	\caption{$N_1$}
	\label{exem2N1}
\end{figure}
\begin{figure}[h!]
	\centering
	\includegraphics[scale=0.4]{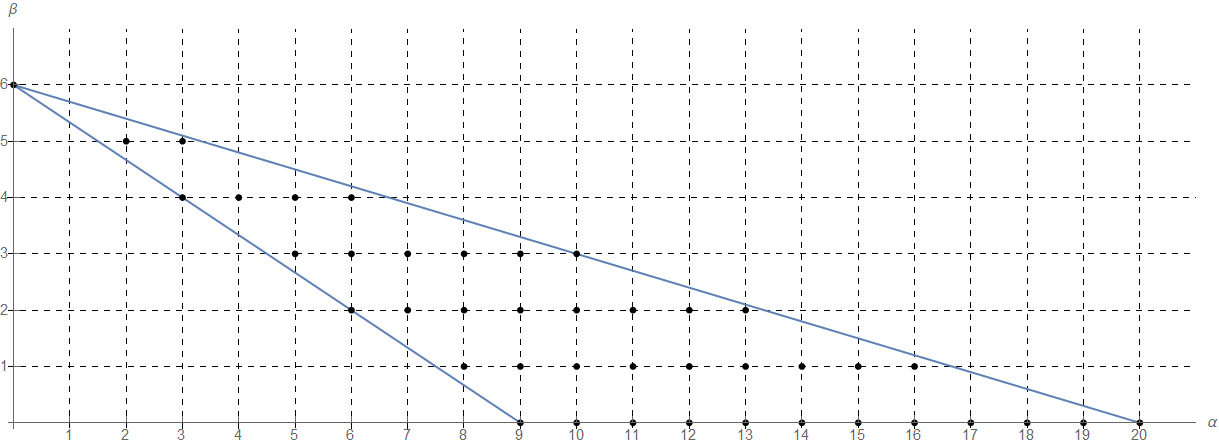}
	\caption{$N_2$}
	\label{exem2N2}
\end{figure}
\end{example}

\noindent
Let $n\in\mathbb N$. We say that $g(x,y)=y^n+\sum_{\ell=0}^{n-1}a_{\ell}(x)y^{\ell}\in\mathbb C[[x]][y]$ is a \em Weierstrass polynomial in the variable \em $y$ if $a_{\ell}(0)=0$, for all $\ell\in\{0,\ldots,n-1\}$. We call $n$ the \em degree \em of $g$.

\noindent
Let $\mathbb R^2_{\geq 0}$ be the first quadrant of $\mathbb R^2$. Let $\widetilde{N}_i$ be the \em Newton polygon \em of $f_i$, that is, the convex hull of the set
\[
\textrm{Supp}(f_i)+\mathbb R^2_{\geq 0}.
\]
This definition is the Local Analytic Geometry version of the Newton polytope of Tropical Geometry (See definition 1.1.3. of \cite{MSTY}). Since $Y_i$ is irreducible, with multiplicity $e_i$ and first characteristic exponent $\lambda_1$, $N_i$ and $\widetilde{N}_i$ share a compact face, namely the line segment that unites $(0,e_i)$ and $(0,e_i\lambda_1)$ (see Section 8.3 of \cite{BK}). Hence
\begin{equation}\label{COMPACFACE}
	\{(0,e_i),(e_i\lambda_1,0)\}\subset\textrm{Supp}(f_i) 
\end{equation}
Furthermore, the point $(0,\beta)$ belongs to $N_i$ if and only if $\beta=e_i$. Therefore, there are $a_{\ell}\in\mathbb C[x]$, $\ell\in\{0,\ldots,e_i-1\}$, and $a\in\mathbb C^{\ast}$ such that
\begin{equation}\label{WEIERPOL}
f_i(x,y)=a y^{e_i}+\sum_{\ell=0}^{e_i-1}a_{\ell}(x)y^{\ell}
\end{equation}
and, for all $\ell\in\{0,\ldots,e_i-1\}$, $a_{\ell}(0)=0$. The notion of degree can be generalized as the greatest power of $y$ that appears on $p\in\mathbb C[[x]][y]$ with non null coefficient. We will denote this notion of degree as $\textrm{deg}_y(p)$. If $p(x,y)$ is the zero polynomial then we set $\textrm{deg}_y(p)=+\infty$. After multiplying $f_i$ by $1/a$, we can assume that $a=1$ and $f_i$ is a Weierstrass polynomial in the variable $y$ of degree $e_i$. From now on, we will always assume that $f_i$ is a Weierstrass polynomial.

\section{Main Result}\label{MAIN}

\noindent
In this section we present our Main Theorem, Theorem \ref{MAINTHEO}. To highlight some key facts of the proof of Theorem \ref{MAINTHEO}, we will begin by presenting an example.

\begin{example}\label{EX1.2}
We return to Example \ref{EX1.1}. The parametrization $\iota$ is defined by
\[
x=t^{12},\, y=c_1 t^{18}+c_2 t^{20}+ c_3 t^{23}.
\]
The induced parametrizations are
\[
\iota_1: x=t^2,\, y=c_1 t^3,\, \iota_2: x=t^{6},\, y=c_1 t^{9}+c_2 t^{10}
\]
and $\iota_3$ is the same as $\iota$. Let $\delta_2(x,y)=-2c_2^3y^3x^5-6c_1^2c_2^3yx^8+c_2^6x^{10}$, $f_1(x,y)= y^2-c_1^2 x^3$ and $f_2(x,y)=f_1^3(x,y)+\delta_2(x,y)$. Note that $\iota_1^{\ast}f_1=0$. We perform the Weierstrass division of $\delta_2$ by $f_1$ elevated to the highest power such that its order is not bigger then the order of $\delta_2$, that is, we divide by $f_1$ and obtain
\[
\delta_2(x,y)=-2c_2^3yx^5f_1(x,y)-8c_1^2c_2^3yx^8+c_2^6x^{10}.
\]
In our case, the remainder of the division has order strictly lower then the order of $f_1$. We will show that this decomposition is unique.
A simple computation in \em Mathematica \em (see \cite{MATH}) or \em Singular \em show us that $\vartheta_{\iota_2}(f_1)=19$, $\iota_2^{\ast}f_2=0$,
\[
\vartheta_{\iota_2}(f_1^3)=\vartheta_{\iota_2}(-8c_1^2c_2^3yx^8+c_2^6x^{10})<\vartheta_{\iota_2}(-2c_2^3yx^5f_1)
\]
and $\vartheta_{\iota_3}(f_2)=117$. On Section \ref{CA} we present an algorithm, based on the proof of Theorem \ref{MAINTHEO}, that allows us to obtain $f_2(x,y)$ from $f_1^3(x,y)$.
\end{example}

\noindent
Let $w\in(x,y)\mathbb C\{x,y\}$, $C=\{w(x,y)=0\}$ a curve and $\tau$ a parametrization induced by a branch of $C$ (see (\ref{PAR})). We can also associate an ideal to the curve $C$ in the ring $\mathbb C[[x,y]]$,  the principal ideal generated by $w$ in $\mathbb C[[x,y]]$, which we denote by $\mathcal I_C \mathbb C[[x,y]]$. Since $\mathbb C\{x,y\}$ is a subring of $\mathbb C[[x,y]]$ there is a canonical injection from $\mathcal I_C \mathbb C\{x,y\}$ to $\mathcal I_C \mathbb C[[x,y]]$. Let $h\in\mathbb C[[x,y]]$, then $h\in \mathcal I_C \mathbb C[[x,y]]$ if and only if $h$ belongs to the kernel of the canonical homomorphism $\mathbb C[[x,y]] \to \sfrac{\mathbb C[[x,y]]}{\mathcal I_C \mathbb C[[x,y]]}$, which is equivalent to $\tau^{\ast}h=0$.

\noindent
Let $\iota$ be as in (\ref{PAR}), $\iota_i$ and $f_i$, $i=1,\ldots,s$ as in Section \ref{NPE}. For $i\in \{1,\ldots,s\}$, let $e_i,\gamma_1^{(i)},\ldots,\gamma^{(i)}_i$ be the generators of the semigroup $\Gamma(\lambda_1,\ldots,\lambda_i)$ as in (\ref{GEN}). Let $f_0(x,y)=y\in\mathbb C[x,y]$.

\begin{theorem}[Main Theorem]\label{MAINTHEO}
For all $i\in\{1,\ldots,s\}$, the following statements hold:
\begin{enumerate}
\item[$(A)$]
There is $\delta_i\in\mathbb C[x,y]$ such that $\textrm{Supp}(\delta_i)\subset N_i\setminus\{(0,e_i)\}$ and $f_i(x,y)=f^{k_i}_{i-1}(x,y)+\delta_i(x,y)$.
\item[$(B)$]
For all $j\in \{i,\ldots,s\}$, $\vartheta_{\iota_j}(f_{i-1})=\gamma^{(j)}_i$ and $\vartheta_{\iota_j}(\partial_yf_{i-1})=\gamma^{(j)}_i-e_j\lambda_i$.
\end{enumerate} 
Furthermore if $g_i$ is a polynomial such that $\iota_i^{\ast}g_i=0$ and $\textrm{Supp}(g_i)\subseteq N_i$ then there is $a\in\mathbb C^{\ast}$ such that $g_i=af_i$.
\end{theorem}

\noindent
The proof of Theorem \ref{MAINTHEO} is split into two subsections. We prove the existence of $f_1$ in Subsection \ref{MAINA} (step $1$ of the proof of Theorem \ref{MAINTHEO}) and that statement $(B)$ holds for $f_1$ in Subsection \ref{MAINB} (Lemma \ref{MAINTHEOI1}). For $i\in\{2,\ldots,s\}$ we apply induction. The proof of the existence of $f_i$ in the terms of statement $(A)$ is done in Subsection \ref{MAINA}, assuming that Theorem \ref{MAINTHEO} holds up to $i-1$. Then, assuming that statement $(A)$ holds up to $i$ and statement $(B)$ holds up to $i-1$, we prove statement $(B)$ for $f_i$ in Subsection \ref{MAINB}. We also prove in Subsections \ref{MAINA} and \ref{MAINB} some results needed for the main proof. Lemma \ref{INDUI} will be particularly useful for the construction of the Algorithm presented in Section \ref{CA}. 

\noindent
The following Corollary shows that we can extend $f_s$ obtained in Theorem \ref{MAINTHEO} to a representative of our plane curve. We can see $f_s$ as an approximation of $f$. Of course the approximation is better the higher the degree of $\iota_s^{\ast}y$ is.
We will not go into details in the proof of Corollary \ref{APPROX} as it repeats many of the reasoning's to be made in the proof of Theorem \ref{MAINTHEO}, statement \em (A), \em and we will present it at the end of Subsection \ref{MAINA}.
\noindent
\begin{corollary}\label{APPROX}
Let $f\in\mathbb C\{x,y\}$, $Y=\{f(x,y)=0\}$ be the germ of a plane curve, $\iota$ as in (\ref{PAR}) and $f_s(x,y)$ as in Theorem \ref{MAINTHEO}. Then there is $\delta\in\mathbb C\{x,y\}$ such that $\vartheta_{\iota}(f_s)=\vartheta_{\iota}(\delta)$ and $f(x,y)=f_s(x,y)+\delta(x,y)$..
\end{corollary}

\subsection{Proof of Statement $\mathbf{(A)}$ of Theorem \ref{MAINTHEO}}\label{MAINA}

\noindent
The following two results are needed for the proof of statement $(A)$ of Theorem \ref{MAINTHEO}.

\noindent
\begin{lemma}\label{LOWERMULT}
Let $m(x_0,\ldots,x_{i-1})=\sum_{\ell=0}^{i-1}e_{\ell} x_{\ell}$ and
\[
\mathcal R=\{(x_0,\ldots,x_{i-1})\in\mathbb R^{i}:(x_0,\ldots,x_{i-1})\geq (0,\ldots,0)\wedge \sum_{\ell=0}^{i-1}\gamma^{(i)}_{\ell+1} x_{\ell}\geq k_i\gamma^{(i)}_{i}\}.
\]	
Then $m(\mathcal R)$ admits minimum, equal to $e_{i}$, and this value is reached only at the point $(0,\ldots,0,k_i)$.
\end{lemma}

\begin{proof}
Note that $m(0,\ldots,0,k_i)=k_i e_{i-1}=e_i$. Let $(h_0,\ldots,h_{i-1})\in\mathbb R^{i}$ such that
\[
(h_0,\ldots,h_{i-1})+(0,\ldots,0,k_i)\in\mathcal R.
\]
Then $h_{\ell}\geq 0$, for all $\ell\in\{0,\ldots,i-2\}$, $h_{i-1}+k_i\geq 0$ and
\begin{equation}\label{INEQMIN1}
	(h_{i-1}+k_i)\gamma^{(i)}_{i}+\sum_{\ell=0}^{i-2}\gamma^{(i)}_{\ell+1} x_{\ell}\geq k_i\gamma^{(i)}_{i}.
\end{equation}
The inequality (\ref{INEQMIN1}) is equivalent to
\begin{equation}\label{INEQMIN2}
	h_{i-1}\geq-\sum_{\ell=0}^{i-2}\frac{\gamma^{(i)}_{\ell+1}}{\gamma^{(i)}_{i}} h_{\ell}.
\end{equation}
Hence
\begin{align*}
		m(h_1,\ldots,h_{i-2},h_{i-1}+k_i)=(h_{i-1}+k_i)e_{i-1}+\sum_{\ell=0}^{i-2}e_{\ell} h_{\ell}\geq & \\
		\geq e_i-e_{i-1} \sum_{\ell=0}^{i-2}\frac{\gamma^{(i)}_{\ell+1}}{\gamma^{(i)}_{i}} h_{\ell}+\sum_{\ell=0}^{i-2}e_{\ell} h_{\ell}=
		e_i+\sum_{\ell=0}^{i-2}\left(e_{\ell}-\frac{\gamma^{(i)}_{\ell+1}}{\gamma^{(i)}_{i}}e_{i-1}\right) h_{\ell}. &
\end{align*}
By statement $(D)$ of Proposition \ref{EQINEQ}, for all $\ell\in\{0,\ldots,i-2\}$,
\[
e_{\ell}-\frac{\gamma^{(i)}_{\ell+1}}{\gamma^{(i)}_{i}}e_{i-1}>0.
\]
Hence, if there is $\ell\in\{0,\ldots,i-2\}$ such that $h_{\ell}>0$, we conclude that
\[
m(h_1,\ldots,h_{i-2},h_{i-1}+k_i)>e_i.
\]
Assume that $h_1=\cdots=h_{i-2}=0$. Then inequality (\ref{INEQMIN2}) implies $h_{i-1}\geq 0$. If $h_{i-1}>0$, we have
\[
m(0,\ldots,0,h_{i-1}+k_i)=h_{i-1}e_{i-1}+e_i>e_i.
\]
We conclude that the result holds.
\end{proof}

\begin{proposition}\label{VALGROW}
	Let $i\in\{1,\ldots,s\}$ and $(w_j)_{j\in \mathbb N_0}$ be a sequence of elements of $\mathbb C[[x,y]]$ such that, for all $j\in\mathbb N_0$, $\vartheta_{\iota_i}(w_{j+1})>\vartheta_{\iota_i}(w_{j})$. Then
	\[
	\vartheta_{\iota_i}\left(\sum_{j\geq 0}w_{j}\right)=\vartheta_{\iota_i}(w_{0}).
	\]
\end{proposition}

\begin{proof}
	Set $\kappa_j=\vartheta_{\iota_i}(w_{j})$. By the definition of $\vartheta_{\iota_i}$, there is $u_j\in\mathbb C[[t]]$ such that $u_j(0)\neq 0$ and $\iota_{i}^{\ast}w_j=t^{\kappa_j}u_j(t)$. Then
	\[
	\iota_{i}^{\ast}\sum_{j\geq 0}w_{j}=\sum_{j\geq 0}t^{\kappa_j}u_j(t)=t^{\kappa_0}\left(u_0(t)+\sum_{j\geq 1}t^{\kappa_j-\kappa_0}u_j(t)\right).
	\] 
	Set 
	\[
	\theta(t)=u_0(t)+\sum_{j\geq 1}t^{\kappa_j-\kappa_0}u_j(t)
	\]
	Since $\kappa_{j+1}>\kappa_j$ then $\theta\in\mathbb C[[t]]$ and $\theta(0)=u_0(0)\neq 0$. Hence the result follows from the definition of $\vartheta_{\iota_i}$.
\end{proof}

\begin{proof}[Proof of Statement $(A)$ of Theorem \ref{MAINTHEO}]
	
\noindent 
\em Step 1: \em Assume $i=1$. 
	
\noindent 
We have
\begin{equation}\label{MINVALN1}
	\min_{(\alpha,\beta)\in N_1}\{\vartheta_{\iota_1}(x^{\alpha}y^{\beta})\}=\min\{k_1\alpha+(k_1\lambda_1)\beta:(\alpha,\beta)\in N_1\}=k_1(k_1\lambda_1)
\end{equation}
and
\[
\max_{(\alpha,\beta)\in N_1}\{\vartheta_{\iota_1}(x^{\alpha}y^{\beta})\}=k_1(k_1\mu_1).
\]
The line segment
\[
\mathcal L=\{(\alpha,\beta)\in\mathbb N_0^{2}:k_1\alpha+(k_1\lambda_1)\beta=k_1(k_1\lambda_1)\}
\]
is a common face to $N_1$ and $\widetilde{N}_1$. Since $k_1$ is the smallest positive integer such that $k_1\lambda_1\in\mathbb Z$ then the only integer points of $\mathcal L$ are $(0,k_1)$ and $(k_1\lambda_1,0)$. Hence there are $a_{\alpha,\beta}\in\mathbb C$, $(\alpha,\beta)\in N_1$, such that $a_{0,k_1},a_{k_1\lambda_1,0}\in\mathbb C^{\ast}$ and
\[
f_1(x,y)=a_{0,k_1}y^{k_1}+a_{k_1\lambda_1,0}x^{k_1\lambda_1}+\sum_{\ell=k_1(k_1\lambda_1)+1}^{k_1(k_1\mu_1)}\left(\sum_{\substack{k_1\alpha+(k_1\lambda_1)\beta=\ell\\(\alpha,\beta)\in N_1}}a_{\alpha,\beta}x^{\alpha}y^{\beta}\right).
\]
Equality $\iota_1^{\ast}f_1=0$ and (\ref{MINVALN1}) imply that 
\[
a_{0,k_1}c_1^{k_1}+a_{k_1\lambda_1,0}=0.
\]
Choose $a_{0,k_1}=1$ and $a_{k_1\lambda_1,0}=-c_1^{k_1}$. Hence there is $\widetilde{\delta}_1\in\mathbb C[x,y]$ such that $\textrm{Supp}(\widetilde{\delta}_1)\subset N_1\setminus \mathcal L$, $\vartheta_{\iota_1}(\widetilde{\delta}_1)>k_1(k_1\lambda_1)$ and
\[
f_1(x,y)=y^{k_1}-c_1^{k_1}x^{k_1\lambda_1}+\widetilde{\delta}_1(x,y).
\]
Set $\delta_1(x,y)=-c_1^{k_1}x^{k_1\lambda_1}+\widetilde{\delta}_1(x,y)$.
	
\vspace{0.5cm}
	
\noindent
Let $i\in\{2,\ldots,s-1\}$. Assume that Theorem \ref{MAINTHEO} holds for $i-1$.
	
\noindent
\em Step 2 (Formal elimination Process): \em We prove the existence of $\widetilde{\delta}_i\in\mathbb C[[x,y]]$ such that $f^{k_i}_{i-1}(x,y)+\widetilde{\delta}_i(x,y)\in \mathcal I_C \mathbb C[[x,y]]$.
	
\noindent
Set $\widetilde{f}_{i,0}(x,y)=f^{k_i}_{i-1}(x,y)$. By the induction hypothesis of statement $(B)$ of Theorem \ref{MAINTHEO},
\begin{equation}\label{VALF0}
	\vartheta_{\iota_i}(\widetilde{f}_{i,0}(x,y))=k_i\gamma^{(i)}_i.
\end{equation}
Furthermore,  since $\textrm{ord}(f_{i-1})=e_{i-1}$ and $C_0(Y_{i-1})=\{y=0\}$ then the order of $\widetilde{f}_{i,0}(x,y)$ is $e_{i}$ and $C_0(\{\widetilde{f}_{i,0}(x,y)=0\})=\{y=0\}$. We prove the existence of $\widetilde{\delta}_i$ using a formal elimination process that starts with $\widetilde{f}_{i,0}$.
	
\noindent
Let $j\in\mathbb N_0$. If $j\neq 0$, assume that $I=\{1,\ldots,j\}$ and that, for all $\ell\in I$, there are $\widetilde{f}_{i,\ell},\widetilde{\delta}_{i,\ell}\in\mathbb C[[x,y]]$ such that
\begin{equation}\label{ORD}
	\textrm{ord}(\widetilde{f}_{i,\ell})=e_i,\, C_0(\{\widetilde{f}_{i,\ell}(x,y)=0\})=\{y=0\},
\end{equation}
\begin{equation}\label{INCVAL}
	\vartheta_{\iota_i}(\widetilde{f}_{i,\ell-1})<\vartheta_{\iota_i}(\widetilde{f}_{i,\ell})\neq  +\infty,
\end{equation}
\begin{equation}\label{DELTACONS}
	\widetilde{f}_{i,\ell}(x,y)-\widetilde{f}_{i,\ell-1}(x,y)=\widetilde{\delta}_{i,\ell}(x,y) \textrm{ and } \vartheta_{\iota_i}(\widetilde{\delta}_{i,\ell})=\vartheta_{\iota_i}(\widetilde{f}_{i,\ell-1}).
\end{equation}
By (\ref{VALSEMIG}), $\vartheta_{\iota_i}(\widetilde{f}_{i,j})\in \Gamma(\lambda_1,\ldots,\lambda_i)$. Hence, by the induction hypothesis, there are non negative integers $\alpha^{(j+1)}$, $\beta^{(j+1)}_0$,$\ldots$, $\beta^{(j+1)}_{i-1}$ such that
\[
	\vartheta_{\iota_i}(\widetilde{f}_{i,j})=	\vartheta_{\iota_i}\left(x^{\alpha^{(j+1)}}f_0^{\beta^{(j+1)}_0}\cdots f_{i-1}^{\beta^{(j+1)}_{i-1}}\right)=e_i\alpha^{(j+1)}+\sum_{\ell=0}^{i-1}\gamma^{(i)}_{\ell+1}\beta^{(j+1)}_{\ell}.
\]
If $j=0$ we will choose $\alpha^{(1)}$, $\beta^{(1)}_0$,$\ldots$, $\beta^{(1)}_{i-1}$ such that $\beta^{(1)}_{i-1}=0$. We can make this choice because of equality (\ref{VALF0}) and statement $(C)$ of Proposition \ref{EQINEQ}. Set $I=\{1,\ldots,j,j+1\}$,
\[
\widetilde{\delta}_{i,j+1}(x,y)= x^{\alpha^{(1)}}f_0^{\beta^{(1)}_0}(x,y)\cdots f_{i-2}^{\beta^{(1)}_{i-2}}(x,y)
\]
and $\widetilde{f}_{i,j+1}(x,y)=\widetilde{f}_{i,j}(x,y)+a_{j+1}\widetilde{\delta}_{i,j+1}(x,y)$, where $a_{j+1}\in\mathbb C$. We now prove that $a_{j+1}\in\mathbb C^{\ast}$ and that $\widetilde{f}_{i,j+1}(x,y)$ is not the null function by showing that
\begin{equation}\label{SUPPDELTA1}
	\textrm{Supp}(\widetilde{f}_{i,j+1})\neq \emptyset.
\end{equation}

\noindent
Set $\kappa=\vartheta_{\iota_i}(\widetilde{f}_{i,j})=\vartheta_{\iota_i}(\widetilde{\delta}_{i,j+1})$. There are $u,\theta\in\mathbb C[t]$ such that $u(0)\neq 0$, $\theta(0)\neq 0$, $\iota^{\ast}_i\widetilde{f}_{i,j}=t^{\kappa}u(t)$ and $\iota^{\ast}_i\widetilde{\delta}_{i,j+1}=t^{\kappa}\theta(t)$. Then
\begin{align*}
	\iota^{\ast}_i\widetilde{f}_{i,j+1}&=t^{\kappa}(u(t)+a_{j+1}\theta(t))=\\
	&=t^{\kappa}(u(0)+a_{j+1}\theta(0))+t^{\kappa}(u(t)-u(0)+a_{j+1}(\theta(t)-\theta(0))).
\end{align*}
Note that $u(t)-u(0)+a_{j+1}(\theta(t)-\theta(0))\in (t)$. Since $u(0)$ and $\theta(0)$ are non null, the equation $u(0)+a_{j+1}\theta(0)=0$ has one solution with respect to $a_{j+1}$ and the solution is non null. Choose $a_{j+1}$ as that solution. Then
\[
\vartheta_{\iota_i}(\widetilde{f}_{i,j+1})\geq \kappa+1> \vartheta_{\iota_i}(\widetilde{f}_{i,j}).
\]
Since ${(0,e_i)}\in \textrm{Supp}(\widetilde{f}_{i,j})$, to prove (\ref{SUPPDELTA1}) it suffices to prove that ${(0,e_i)}\not\in \textrm{Supp}(\widetilde{\delta}_{i,j+1})$. If $\alpha^{(j+1)}\neq 0$ then $\textrm{Supp}(\widetilde{\delta}_{i,j+1})\subseteq \{(\alpha,\beta)\in\mathbb N_0^2:\alpha\neq 0\}$ and we conclude that ${(0,e_i)}\not\in \textrm{Supp}(\widetilde{\delta}_{i,j+1})$. Assume $\alpha^{(j+1)}=0$. Set
\begin{equation}\label{MULTF}
	m(\beta_0,\ldots,\beta_{i-1})=\sum_{\ell=0}^{i-1}e_{\ell}\beta_{\ell}
\end{equation}
and
\begin{equation}\label{REGIONCONST}
	\mathcal R=\{(\beta_0,\ldots,\beta_{i-1})\in\mathbb R^{i}:(\beta_0,\ldots,\beta_{i-1})\geq (0,\ldots,0)\wedge \sum_{\ell=0}^{i-1}\gamma^{(i)}_{\ell+1} \beta_{\ell}\geq k_i\gamma^{(i)}_{i}\}.
\end{equation}
Note that since $\textrm{ord}(f_{\ell})=e_{\ell}$, the order of $\widetilde{\delta}_{i,j+1}$ is $m(\beta_0^{(j+1)},\ldots,\beta_{i-1}^{(j+1)})$. If $j=0$, Lemma \ref{LOWERMULT} implies that
\[
m(\beta_0^{(1)},\ldots,\beta_{i-2}^{(1)},0)>e_i.
\]
If $j\neq 0$, inequality
\[
\sum_{\ell=0}^{i-1}\gamma^{(i)}_{\ell+1}\beta^{(j+1)}_{\ell}=\vartheta_{\iota_i}(f_{i,j})>k_i\gamma^{(i)}_i
\]
and Lemma \ref{LOWERMULT} imply that $m(\beta_0^{(j+1)},\ldots,\beta_{i-1}^{(j+1)})>e_i$. We conclude on both cases that ${(0,e_i)}\not\in \textrm{Supp}(\widetilde{\delta}_{i,1})$, otherwise we would have $m(\beta_0^{(j+1)},\ldots,\beta_{i-1}^{(j+1)})\leq e_i$.

\noindent
If $\iota_i^{\ast}\widetilde{f}_{i,j+1}=0$ set $\widetilde{f}_i(x,y)=\widetilde{f}_{i,j+1}(x,y)$. If $\vartheta_{\iota_i}(\widetilde{f}_{i,j+1})\neq +\infty$, we iterate the procedure. If this procedure does not end in a finite number of steps, then $I=\mathbb N$ and
\begin{equation}\label{FORMALFI}
	\widetilde{f}_i(x,y)=f^{k_i}_{i-1}(x,y)+\sum_{\ell\geq 1}a_{\ell}\widetilde{\delta}_{i,\ell}.
\end{equation}
All that is left to prove is that in this case $\iota_i^{\ast}\widetilde{f}_{i}=0$. Let $\ell\in \mathbb N$. By (\ref{INCVAL}), $(\vartheta_{\iota_i}(\widetilde{f}_{i,j}))_{j\in\mathbb N_0}$ is a strictly increasing sequence of positive integers. Hence we can choose $j$ large enough such that
$\vartheta_{\iota_i}(\widetilde{f}_{i,j})>\ell$. By (\ref{DELTACONS}) we can rewrite equality (\ref{FORMALFI}) as
\[
\widetilde{f}_i(x,y)=\widetilde{f}_{i,j}(x,y)+\sum_{\ell\geq j+1}a_{\ell}\widetilde{\delta}_{i,\ell}.
\]
Therefore
\[
\vartheta_{\iota_i}(\widetilde{f}_i)\geq \min\left\{\vartheta_{\iota_i}(\widetilde{f}_{i,j}),\vartheta_{\iota_i}\left(\sum_{\ell\geq j+1}a_{\ell}\widetilde{\delta}_{i,\ell}\right)\right\}.
\]
By (\ref{INCVAL}) and (\ref{DELTACONS}), $(\vartheta_{\iota_i}(\widetilde{\delta}_{i,j}))_{j\in\mathbb N_0}$ is also a strictly increasing sequence. Therefore, applying Proposition \ref{VALGROW}, we obtain
\[
\vartheta_{\iota_i}\left(\sum_{\ell\geq j+1}a_{\ell}\widetilde{\delta}_{i,\ell}\right)=\vartheta_{\iota_i}(\widetilde{\delta}_{i,j+1})=\vartheta_{\iota_i}(\widetilde{f}_{i,j}).
\]
Hence
\[
\vartheta_{\iota_i}(\widetilde{f}_i)\geq \vartheta_{\iota_i}(\widetilde{f}_{i,j})>\ell
\]
and we conclude that $\vartheta_{\iota_i}(\widetilde{f}_i)=+\infty$, that is, $\iota_i^{\ast}\widetilde{f}_{i}=0$. Set $\widetilde{\delta}_i(x,y)=\sum_{\ell\in I}a_{\ell}\widetilde{\delta}_{i,\ell}(x,y)$. As seen in the construction of the $\widetilde{\delta}_{i,\ell}$'s, for all $\ell\in I$, ${(0,e_i)}\not\in \textrm{Supp}(\widetilde{\delta}_{i,j+1})$. Therefore ${(0,e_i)}\not\in \textrm{Supp}(\widetilde{\delta}_i)$.
	
\vspace{0.5cm}
	
\noindent
\em Step 3: \em We prove the existence of $a\in\mathbb C^{\ast}$ and $\delta_i\in\mathbb C[x,y]$ such that $\textrm{Supp}(\delta_i)\subseteq N_i\setminus \{(0,e_i)\}$ and $f_i(x,y)=af^{k_i}_{i-1}(x,y)+\delta_i(x,y)$.
	
\noindent
Let  $\widetilde{f}_i(x,y)=f^{k_i}_{i-1}(x,y)+\widetilde{\delta}_i(x,y)$ be as in step 2 and $f_i$ as in Section \ref{NPE}. Since $\widetilde{f}_i\in I_C \mathbb C[[x,y]]=(f_i)\mathbb C[[x,y]]$ then there is $u\in\mathbb C[[x,y]]$ such that
\begin{equation}\label{FORMALREL}
	f_i(x,y)=u(x,y)\widetilde{f}_i(x,y).
\end{equation}
By construction $\textrm{ord}(\widetilde{f}_i)=e_i$. As stated in Section \ref{NPE}, $\textrm{ord}(f_i)=e_i$. Therefore $\textrm{ord}(u)=0$ and we conclude that $u$ is a unit of $\mathbb C[[x,y]]$. Hence there is $a\in\mathbb C^{\ast}$ and $\varepsilon\in (x,y)\mathbb C[[x,y]]$ such that $u(x,y)=a+\varepsilon(x,y)$. We rewrite equality (\ref{FORMALREL}) as
\[ f_i(x,y)=af^{k_i}_{i-1}(x,y)+[a\widetilde{\delta}_i(x,y)+\varepsilon(x,y)(f^{k_i}_{i-1}(x,y)+\widetilde{\delta}_i(x,y))].
\]
Set $\delta_i(x,y)=f_i(x,y)-af^{k_i}_{i-1}(x,y)$. Since $f_i,f^{k_i}_{i-1}\in\mathbb C[x,y]$, we conclude that $\delta_i\in\mathbb C[x,y]$. To prove that $\textrm{Supp}(\delta_i)\subseteq N_i\setminus\{(0,e_i)\}$, given that $\textrm{Supp}(f_i)\subseteq N_i$, it suffices to prove that  $\textrm{Supp}(f^{k_i}_{i-1})\subseteq N_i$ and $\{(0,e_i)\}\not\in \textrm{Supp}(\delta_i)$. The Multinomial formula implies
\[
\textrm{Supp}(f_{i-1}^{k_i})\subseteq \sum_{j=1}^{k_i}\textrm{Supp}(f_{i-1}).
\]
By the induction hypothesis, $\textrm{Supp}(f_{i-1})\subseteq N_{i-1}$, hence
\[
\sum_{j=1}^{k_i}\textrm{Supp}(f_{i-1})\subseteq \sum_{j=1}^{k_i}N_{i-1}.
\]
Let $(\alpha_j,\beta_j)\in N_{i-1}$, $j=1,\ldots,k_i$. By (\ref{Ni}), the inequality 
\begin{equation}\label{SMALLNi1}
	k_1\sum_{j=1}^{k_i}\alpha_j+(k_1\lambda_1)\sum_{j=1}^{k_i}\beta_j=\sum_{j=1}^{k_i}\left(k_1\alpha_j+(k_1\lambda_1)\beta_j\right)\geq k_i e_{i-1}(k_1\lambda_1)=e_i(k_1\lambda_1)
\end{equation}
and inequality
\begin{align}\label{SMALLNi2}
	e_{i}\sum_{j=1}^{k_i}\alpha_j+(e_{i}\mu_{i-1})\sum_{j=1}^{k_i}\beta_j&=k_i\sum_{j=1}^{k_i}\left(e_{i-1}\alpha_j+(e_{i-1}\mu_{i-1})\beta_j\right)\leq \\ 
	&\leq k^2_i e_{i-1}(e_{i-1}\mu_{i-1})=e_i(e_{i}\mu_{i-1}) \nonumber
\end{align}
hold. Set
\[
\widetilde{N}_i=\{(\alpha,\beta)\in\mathbb N_0^2:k_1\alpha+(k_1\lambda_1)\beta\geq  e_i(k_1\lambda_1) \wedge e_{i}\alpha+(e_{i}\mu_{i-1})\beta\leq e_i(e_{i}\mu_{i-1}) \}.
\]
Since $\mu_{i-1}< \mu_i$ then $\widetilde{N}_i\subset N_i$. From inequalities (\ref{SMALLNi1}) and (\ref{SMALLNi2}) we conclude that 
\[
\sum_{j=1}^{k_i}N_{i-1}\subseteq \widetilde{N}_i.
\]
Note that $\textrm{Supp}(f_i)\subseteq N_i$ and $\varepsilon\in (x,y)\mathbb C[[x,y]]$ imply that $\{(0,e_i)\}\not\in \textrm{Supp}(\varepsilon f^{k_i}_{i-1})$. Also, by construction, $\{(0,e_i)\}\not\in \textrm{Supp}(a\widetilde{\delta}_i)$. Therefore, taking into account the definition of $\delta_i$, to prove that $\{(0,e_i)\}\not\in \textrm{Supp}(\delta_i)$ it suffices to prove that $\{(0,e_i)\}\not\in \textrm{Supp}(\varepsilon\widetilde{\delta}_i)$. This fact results immediately that for all $\ell\in I$, as proved in the construction of $\widetilde{\delta}_{i,\ell}$, if $\textrm{Supp}(\widetilde{\delta}_{i,\ell})\not \subseteq \{(\alpha,\beta)\in\mathbb N_0^2:\alpha\neq 0\}$ then $\textrm{ord}(\widetilde{\delta}_{i,\ell})>e_i$.

\vspace{0.5cm}
\noindent
\em Step 4: \em The polynomial $f_i$ is unique up to  multiplication by a non null constant.
	
\noindent
Let $g_i\in\mathbb C[x,y]$ such that $\iota_i^{\ast}g_i=0$ and $\textrm{Supp}(g_i)\subseteq N_i$. Since $\mathcal I_{Y_i}\mathbb C\{x,y\}=(f_i)\mathbb C\{x,y\}=(g_i)\mathbb C\{x,y\}$, there is $a\in\mathbb C^{\ast}$ and $\varepsilon\in (x,y)\mathbb C\{x,y\}$ such that
\begin{equation}\label{UNICITY}
	g_i(x,y)=(a+\varepsilon(x,y))f_i(x,y).
\end{equation}
Assume $\varepsilon$ non null. Let $n=\textrm{ord}(\varepsilon)$. There are $\theta\in(x,y)^{e_i+1}\mathbb C\{x,y\}$, $\varphi\in(x,y)^{n+1}\mathbb C\{x,y\}$ and a homogeneous polynomial $\widetilde{\varepsilon}$ of order $n$, such that $f_i(x,y)=y^{e_i}+\theta(x,y)$ and $\varepsilon(x,y)=\widetilde{\varepsilon}(x,y)+\varphi(x,y)$. We rewrite equality (\ref{UNICITY}) as
\[
g_i(x,y)=af_i(x,y)+y^{e_i}\widetilde{\varepsilon}(x,y)+y^{e_i}\varphi(x,y)+\theta(x,y)\varphi(x,y).
\]
By definition $y^{e_i}\widetilde{\varepsilon}$ is a homogenous polynomial of order $e_i+n$, $\textrm{ord}(y^{e_i}\varphi)=e_i+n+1$ and $\textrm{ord}(\theta\varphi)=e_i+n+2$. Then 
\[
\textrm{Supp}(y^{e_i}\widetilde{\varepsilon})\subseteq \textrm{Supp}(y^{e_i}\widetilde{\varepsilon}+y^{e_i}\varphi+\theta\varphi)
\]
The set $\textrm{Supp}(y^{e_i}\widetilde{\varepsilon})$ is equal to
\[
\{(0,e_i)\}+\textrm{Supp}(\widetilde{\varepsilon})
\]
and thus
\[
N_i\cap \textrm{Supp}(y^{e_i}\widetilde{\varepsilon})=\emptyset.
\]
Since $\textrm{Supp}(f_i)\subset N_i$, we conclude that 
\[
\textrm{Supp}(y^{e_i}\widetilde{\varepsilon})\subseteq \textrm{Supp}(g_i),
\]
which contradicts the fact that $\textrm{Supp}(g_i)\subseteq N_i$. Therefore $\varepsilon$ must be the null series.
\end{proof}

\begin{proof}[Proof of Corollary \ref{APPROX}]
If $\iota^{\ast}f_s=0$ then $\delta(x,y)=0$ and the result is proved. Otherwise we repeat the formal elimination process of step $2$ to obtain $\widetilde{\delta}\in\mathbb C[[x,y]]$ such that $\iota^{\ast}(f_s+\widetilde{\delta})=0$. Applying a similar reasoning made in step $3$, we construct a $\delta\in\mathbb C\{x,y\}$ such that $f(x,y)=f_s(x,y)+\delta(x,y)$.
\end{proof}

\subsection{Proof of Statement $\mathbf{(B)}$ of Theorem \ref{MAINTHEO}}\label{MAINB}

\noindent
We shall denote the operator partial derivative with respect to the variable $y$ as $\partial_y$. 

\noindent
Proposition \ref{SEMIREL} show us that there is a natural injective map $\Gamma(\lambda_1,\ldots,\lambda_{i-1})$ to $\Gamma(\lambda_1,\ldots,\lambda_i)$ that maps $n$ into $\frac{e_i}{e_{i-1}}n$. One can ask if there is such a map at valuation level, in the sense that, given $g\in\mathbb C[[x,y]]$, $\vartheta_{\iota_{i}}(g)=\frac{e_i}{e_{i-1}}\vartheta_{\iota_{i-1}}(g)$. The answer is not so simple as it can be seen in Example \ref{EX1.2}, where $\vartheta_{\iota_1}(f_2)=19$,  $\vartheta_{\iota_2}(f_2)=+\infty$ and $\vartheta_{\iota_3}(f_2)=117$. We inspect some simple cases first. Let $g(x)\in\mathbb C[[x]]$. Then, for all $i\in\{1,\ldots,s\}$ and $j\in\{i,\ldots,s\}$,
\begin{equation}\label{VALX}
	\vartheta_{\iota_j}(g)=e_j\textrm{ord}(g)=\frac{e_j}{e_i}e_i\textrm{ord}(g)=\frac{e_j}{e_i}\vartheta_{\iota_i}(g).
\end{equation}
For all $i\in\{1,\ldots,s\}$, $j\in\{i,\ldots,s\}$, $d\in\mathbb N_0$ and $\ell\in\{0,\dots,d\}$,
\begin{align}\label{VALY}
\vartheta_{\iota_j}(\partial_y^{\ell} y^d)&=(d-\ell)\vartheta_{\iota_j}( y)=(d-\ell)\textrm{ord}(\iota_j^{\ast}y)=(d-\ell)e_j\lambda_1=\\
&=(d-\ell)\frac{e_j}{e_i}e_i\lambda_1=\frac{e_j}{e_i}\vartheta_{\iota_i}(\partial_y^{\ell} y^d).\nonumber
\end{align}

\noindent
Let $n\in\mathbb N$ and $U$ be an arbitrary small open disc of $\mathbb C$ centered in zero. For $t\in U$, define the map $\chi_n(t)=t^n$. Let $g\in\mathbb C[[x,y]]$ and $h\in\mathbb C[[t]]$. Note that $\textrm{ord}(\chi_n^{\ast}h)=n\,\textrm{ord}(h)$ and if, for some $i\in\{1,\ldots,s\}$, $\iota_i^{\ast}g=0$ then $(\iota_i\circ \chi_n)^{\ast}g=0$.

\begin{lemma}\label{INDUI1}
Let $\alpha,\beta\in\mathbb N_0$, $a,b\in\mathbb C[[x]]$ and 
\[
g(x,y)=\sum_{\ell=0}^{n}a_{\ell}(x)y^{\ell}\in\mathbb C[[x]][y]
\]
such that $0\leq n<e_1$ and $\textrm{deg}_y(g)=n$. The following statements hold:
\begin{enumerate}
	\item [$(A)$]
	If $\vartheta_{\iota_1}(ay^{\alpha})=\vartheta_{\iota_1}(ay^{\beta})$ then $\alpha-\beta$ is a multiple of $e_1$.
	\item [$(B)$]
    There is one and only one $\tau_0\in\{0,\ldots,\textrm{deg}_y(g)\}$ such that 
    \[
    \vartheta_{\iota_1}(g)=\vartheta_{\iota_1}(a_{\tau_0}(x)y^{\tau_0}).
    \]
    Furthermore, for all $j\in \{1,\ldots,s\}$,
	\begin{equation}\label{INDUI1A}
	\vartheta_{\iota_j}(g)=\frac{e_j}{e_1}\vartheta_{\iota_1}(g)
	\end{equation}
 	and, if $\tau_0\neq 0$ then
    \begin{equation}\label{INDUI1B} 
    \vartheta_{\iota_j}(\partial_y g)= \vartheta_{\iota_j}(g)-e_j\lambda_1,
	\end{equation}
	otherwise
	\begin{equation}\label{INDUI1C}
	\vartheta_{\iota_j}(\partial_y g)> \vartheta_{\iota_j}(g)-e_j\lambda_1.
	\end{equation}
	\item [$(C)$]
	For all $j\in \{1,\ldots,s\}$ and $\sigma\in\mathbb N$,
	\begin{equation}\label{HIGHORDDER}
	 \vartheta_{\iota_j}(\partial^{\sigma}_y g)\geq \vartheta_{\iota_j}(g)-\sigma e_j\lambda_1.
	\end{equation}
\end{enumerate}
\begin{proof}
We begin by proving statement $(A)$. Equality (\ref{VALPROD}) allows us to rewrite $\vartheta_{\iota_1}(ay^{\alpha})=\vartheta_{\iota_1}(ay^{\beta})$ as 
\begin{equation}\label{INDUI1EQ1}
\vartheta_{\iota_1}(a)+\alpha\vartheta_{\iota_1}(y)= \vartheta_{\iota_1}(b)+\beta\vartheta_{\iota_1}(y).
\end{equation}
We can assume $\alpha\geq \beta$. Replacing $\vartheta_{\iota_1}(a)$,  $\vartheta_{\iota_1}(b)$ and $\vartheta_{\iota_1}(y)$ by their respective values in (\ref{INDUI1EQ1}), we obtain
\[
(\alpha-\beta)e_1\lambda_1=e_1 \textrm{ord}(b)-e_1 \textrm{ord}(a).
\]
Hence $\vartheta_{\iota_1}(ay^{\alpha})=\vartheta_{\iota_1}(ay^{\beta})$ if and only if $(\alpha-\beta)\lambda_1=\textrm{ord}(b)-\textrm{ord}(a)$.
Let $q$ and $r\in\{0,\ldots,e_1-1\}$ the unique non negative integers such that
\[
\alpha-\beta=qe_1+r.
\]
Applying this equality in $(\alpha-\beta)\lambda_1=\textrm{ord}(b)-\textrm{ord}(a)$, we obtain
\begin{equation}\label{VALEQI1}
	r\lambda_1=\textrm{ord}(b)-\textrm{ord}(a)-q(e_1\lambda_1)\in\mathbb Z.
\end{equation}
Statement \em C3) \em (see Section \ref{SQO}) implies that (\ref{VALEQI1}) holds if and only if $r=0$.

\noindent
We now prove statement $(B)$. Assume that $n\neq 0$. Set
\begin{equation}
	\rho_{d}=\min\{\vartheta_{\iota_1}(a_{\ell}y^{\ell}):a_{\ell}\neq 0 \textrm{ and } \ell\in\{d,\ldots,n\}\},\, d=0,1.
\end{equation}
Since $n<e_1$, statement \em (A) \em implies that there is one and only one $\tau_d\in\{d,\ldots,n\}$, $d=0,1$, such that
\begin{equation}\label{ORDVALGDG1}
	\vartheta_{\iota_1}(a_{\tau_d}y^{\tau_d})=\rho_d. 
\end{equation}
Hence $\vartheta_{\iota_1}(g)=\rho_0$. If $\tau_0\neq 0$ then $\tau_1=\tau_0$, $\rho_1=\rho_0$ and
\begin{equation}\label{ORDVALGDG2}
\vartheta_{\iota_1}(\partial_y g)=\vartheta_{\iota_1}(a_{\tau_1}y^{\tau_1-1})=\vartheta_{\iota_1}(a_{\tau_1})+(\tau_1-1)e_1\lambda_1=\rho_1-e_1\lambda_1.
\end{equation}
If $\tau_0= 0$ then $\tau_1>\tau_0$, $\rho_1>\rho_0$ and
\begin{equation}\label{ORDVALGDG3}
\vartheta_{\iota_1}(\partial_y g)=\rho_1-e_1\lambda_1>\rho_0-e_1\lambda_1.
\end{equation}
Equalities (\ref{VALX}) and (\ref{VALY}) imply that, for all $j\in\{1,\ldots,s\}$ and $\ell\in\{0,\ldots,n\}$, 
\[
\vartheta_{\iota_j}(a_{\ell}y^{\ell})=\frac{e_j}{e_1}\vartheta_{\iota_1}(a_{\ell}y^{\ell}) \textrm{ and } \vartheta_{\iota_j}\left(\partial_y (a_{\ell}y^{\ell})\right)=\frac{e_j}{e_1}\vartheta_{\iota_1}\left(\partial_y (a_{\ell}y^{\ell})\right).
\]
Hence
\[
\vartheta_{\iota_j}(g)=\frac{e_j}{e_1}\rho_0 \textrm{ and } \vartheta_{\iota_j}(\partial_y g)=\frac{e_j}{e_1}(\rho_1-e_1\lambda_1)=\frac{e_j}{e_1}\rho_1-e_j\lambda_1.
\]
Thus (\ref{INDUI1A}) and (\ref{INDUI1B}) hold. Under the condition of $\tau_0=0$, from inequality (\ref{ORDVALGDG3}) we obtain that
\[
\vartheta_{\iota_j}(\partial_y g)=\frac{e_j}{e_1}(\rho_1-e_1\lambda_1)>\frac{e_j}{e_1}(\rho_0-e_1\lambda_1)=\frac{e_j}{e_1}\rho_0-e_j\lambda_1
\]
and (\ref{INDUI1C}) is proven.

\noindent
The proof of equality (\ref{INDUI1A}) still holds if $n=0$. Inequality (\ref{INDUI1C}) holds trivially if $n=0$.

\noindent
We prove statement $(C)$ by induction on $\sigma\in\{1,\ldots,n\}$. Note that statement $(C)$ for $\sigma=1$ is proved in statement $(B)$ and that, for $\sigma\geq n+1$, inequality (\ref{HIGHORDDER}) holds trivially. Assume that $\sigma\in\{2,\ldots,n\}$. Then $0\leq deg_y\left(\partial^{\sigma-1}_yg\right)<n<e_1$ and by statement $(B)$,
\[
\vartheta_{\iota_j}\left(\partial_y (\partial^{\sigma-1}_yg)\right)\geq  \vartheta_{\iota_j}(\partial^{\sigma-1}_yg)-e_j\lambda_1.
\]
By the induction hypothesis,
\[
\vartheta_{\iota_j}(\partial^{\sigma-1}_y g)\geq \vartheta_{\iota_j}(g)-(\sigma-1) e_j\lambda_1.
\]
Thus
\[
\vartheta_{\iota_j}\left(\partial_y (\partial^{\sigma-1}_yg)\right)\geq  \vartheta_{\iota_j}(g)-\sigma e_j\lambda_1.
\]
\end{proof}

\end{lemma}

\begin{lemma}\label{MAINTHEOI1}
Assume that there is $\delta_1\in\mathbb C[x,y]$ such that $\textrm{Supp}(\delta_1)\subset N_1\setminus\{(0,e_1)\}$ and $f_1(x,y)=y^{e_1}+\delta_1(x,y)$. Then, for all $j\in \{2,\ldots,s\}$, $\vartheta_{\iota_j}(f_1)=\gamma^{(j)}_2$ and $\vartheta_{\iota_j}(\partial_yf_{1})=\gamma^{(j)}_2-e_j\lambda_2$.
\end{lemma}

\begin{proof}
We remind the reader that $e_1=k_1$. Let $\delta_1$ be as in statement \em (A) \em of Theorem \ref{MAINTHEO}. We must have 
\begin{equation}\label{VALf1}
\vartheta_{\iota_1}(\delta_1)=\vartheta_{\iota_1}(y^{e_1}),
\end{equation}
otherwise $\vartheta_{\iota_1}(f_1)=\min\{\vartheta_{\iota_1}(\delta_1),\vartheta_{\iota_1}(y^{e_1})\}\neq +\infty$ which contradicts $\iota_1^{\ast}f_1=0$. Note that $\vartheta_{\iota_1}(y^{e_1})=e_1\gamma_1^{(1)}$. Write $\delta_1$ as an element of $\mathbb C[x][y]$ and let $n=\textrm{deg}_y(\delta_1)$. Then there is $a_{\ell}\in\mathbb C[x]$, $\ell\in\{0,,\ldots,n\}$, such that
\[
\delta_1(x,y)=\sum_{\ell=0}^{n}a_{\ell}(x)y^n.
\]
As $\textrm{Supp}(\delta_1)\subset N_1\setminus\{(0,e_1)\}$ then $0\leq n<e_1$. Thus by statement \em (B) \em of Lemma \ref{INDUI1}, set $\tau_0$ the only element of $\{0,\ldots,n\}$ such that $\vartheta_{\iota_1}(\delta_1)=\vartheta_{\iota_1}(a_{\tau_0}y^{\tau_0})$. Since $\vartheta_{\iota_1}(\delta_1)=\vartheta_{\iota_1}(y^{e_1})$, statement \em (A) \em of Lemma \ref{INDUI1} implies that $e_1-\tau_0$ is a multiple of $e_1$. Given that $\tau_0\in\{0,\ldots,e_1-1\}$ then we must have $\tau_0=0$.

\noindent
Let $j\in\{2,\ldots,s\}$. Set
\[
\epsilon_{1,j}=c_2t^{e_j\lambda_2}+\varphi_2(t^{e_j/k})+\cdots+c_jt^{e_j\lambda_j}+\varphi_j(t^{e_j/k}).
\]
Therefore
\[
\iota_{j}^{\ast}x=(\iota_1\circ \chi_{\frac{e_j}{e_1}})^{\ast}x \textrm{ and } \iota_{j}^{\ast}y=(\iota_1\circ \chi_{\frac{e_j}{e_1}})^{\ast}y+\epsilon_{1,j}.
\]
Note that $\textrm{ord}(\epsilon_{1,j})=e_j\lambda_2$. By Taylor's formula,
\begin{align}\label{TAYLORF1}
	\iota_j^{\ast}f_1&=(\iota_1\circ \chi_{\frac{e_j}{e_1}})^{\ast}f_1(x,y)+k_1\epsilon_{1,j} (\iota_1\circ \chi_{\frac{e_j}{e_1}})^{\ast}y^{k_1-1}+\\	&+\sum_{\ell=2}^{e_1}\frac{k_1!}{(k_1-\ell)!\ell!}\epsilon_{1,j}^{\ell}(\iota_1\circ \chi_{\frac{e_j}{e_1}})^{\ast}y^{k_1-\ell}+
	\sum_{\ell=1}^{n}\frac{1}{\ell!}\epsilon_{1,j}^{\ell}(\iota_1\circ \chi_{\frac{e_j}{e_1}})^{\ast}\partial^{\ell}_y \delta_1(x,y).\nonumber
\end{align}
Since $\iota_1^{\ast}f_1=0$ then 
\begin{equation}\label{EQVALF1}
	(\iota_1\circ \chi_{\frac{e_j}{e_1}})^{\ast}f_1=0.
\end{equation}
For all $\ell\in\{2,\ldots e_1\}$,
\begin{align}\label{INEQ1VALF1}
	\textrm{ord}\left(\epsilon_{1,j}^{\ell}(\iota_1\circ \chi_{\frac{e_j}{e_1}})^{\ast}y^{k_1-\ell}\right)-\textrm{ord}\left(\epsilon_{1,j} (\iota_1\circ \chi_{\frac{e_j}{e_1}})^{\ast}y^{k_1-1}\right)=\\
	\nonumber
	=\ell e_j\lambda_2+\frac{e_j}{e_1}(k_1-\ell)e_1\lambda_1-\left( e_j\lambda_2+\frac{e_j}{e_1}(k_1-1)e_1\lambda_1\right)=\\
	\nonumber
	=(\ell-1)e_j(\lambda_2-\lambda_1)>0.\hspace{4.75cm}
\end{align}
Inequality (\ref{INDUI1C}) implies that
\begin{equation}\label{INEQ3VALF1}
	\vartheta_{\iota_j}(\partial_y^{\ell} \delta_1)>\vartheta_{\iota_j}(\delta_1)-\ell e_j\lambda_1=(k_1-\ell)e_j\lambda_1.
\end{equation}
Therefore, for all $\ell\in\{1,\ldots n\}$,
\begin{align}\label{INEQ2VALF1}
	\textrm{ord}\left(\epsilon_{1,j}^{\ell}(\iota_1\circ \chi_{\frac{e_j}{e_1}})^{\ast}\partial^{\ell}_y \delta_1\right)-\textrm{ord}\left(\epsilon_{1,j} (\iota_1\circ \chi_{\frac{e_j}{e_1}})^{\ast}y^{k_1-1}\right)>\\
	\nonumber
	>\ell e_j\lambda_2+\frac{e_j}{e_1}(k_1-\ell)e_1\lambda_1-\left( e_j\lambda_2+\frac{e_j}{e_1}(k_1-1)e_1\lambda_1\right)=\\
	\nonumber
	=(\ell-1)e_j(\lambda_2-\lambda_1)\geq 0.\hspace{4.8cm}
\end{align}
Applying equality (\ref{EQVALF1}) and inequalities (\ref{INEQ1VALF1}), (\ref{INEQ2VALF1}), we conclude that
\[
\vartheta_{\iota_j}(f_1)=\textrm{ord}(\iota_j^{\ast}f_1)=\textrm{ord}\left(\epsilon_{1,j} (\iota_1\circ \chi_{\frac{e_j}{e_1}})^{\ast}y^{k_1-1}\right)=e_j\lambda_2+(k_1-1)\gamma_1^{(j)}=\gamma_2^{(j)}.
\]
From (\ref{INEQ3VALF1}) we obtain that 
\[
\vartheta_{\iota_j}(\partial_y\delta_1)>(k_1-1)e_j\lambda_1=\vartheta_{\iota_j}(y^{k_1-1}).
\]
Therefore,
\[
\vartheta_{\iota_j}(\partial_yf_1)=\vartheta_{\iota_j}(y^{k_1-1})=(k_1-1)\gamma_1^{(j)}=\gamma_2^{(j)}-e_j\lambda_2.
\]
\end{proof}

\noindent
To prove a more general version of Lemma \ref{INDUI1} we will need to apply the Weierstrass Division Algorithm, so we need to get (re)acquainted with the algorithm (see Theorem 6.2.6 of \cite{GP2} and its proof for the formal version). We are only interested in the case where we divide $g\in\mathbb C[[x]][y]$ by a Weierstrass polynomial $p$ with degree in $y$ strictly smaller then the degree of $g$ in $y$. Let $n=\textrm{deg}_y(g)$, $m=\textrm{deg}_y(p)$,
\[
g(x,y)=\sum_{\ell=0}^{n}a_{\ell}(x)y^{\ell} \textrm{ and } p(x,y)=y^{m}+\sum_{\ell=0}^{m-1}b_{\ell}(x)y^{\ell}.
\]
Set $w(x,y)=-\sum_{\ell=0}^{m-1}b_{\ell}(x)y^{\ell}$. Given $s(x,y)=\sum_{\ell\geq 0}c_{\ell}(x)y^{\ell}\in \mathbb C[[x]][[y]]$, define
\[
h(s)(x,y)=\frac{1}{y^m}\sum_{\ell\geq m}c_{\ell}(x)y^{\ell}=\sum_{\ell\geq 0}c_{\ell+m}(x)y^{\ell},\, H(s)(x,y)=h(sw)(x,y)
\]
and, for $\ell\in\mathbb N$, $H^{\ell}$ as $H$ composed with itself $\ell$ times. Set
\[
u(x,y)=h(g)(x,y)=\sum_{\ell= 0}^{n-m}a_{\ell+m}(x)y^{\ell}.
\]
The Weierstrass Division Algorithm tells us that there exist uniquely determined $q\in\mathbb C[[x,y]]$ and $r\in\mathbb C[[x]][y]$ such that $g(x,y)=q(x,y)p(x,y)+r(x,y)$ and $\textrm{deg}_y(r)<m$. Furthermore
\begin{equation}\label{QUOTIENT}
	q(x,y)=u(x,y)+\sum_{\ell\geq 1}H^{\ell}(u)(x,y).
\end{equation}

\begin{proposition}\label{WEIERQUOTIENT}
Let $p,g\in\mathbb C[[x]][y]$ such that $p$ is a Weierstrass polynomial, $\textrm{deg}_y(p)=m$ and $\textrm{deg}_y(g)=n$. Assume that $n>m$. Then there exist uniquely determined $q\in\mathbb C[[x]][y]$ and $r\in\mathbb C[[x]][y]$ such that $g(x,y)=q(x,y)p(x,y)+r(x,y)$, $\textrm{deg}_y(r)<m$ and $\textrm{deg}_y(q)<n-m$. Furthermore, if $g$ and $p$ both belong to $\mathbb C[x][y]$ then $q\in\mathbb C[x][y]$.
\end{proposition}

\begin{proof}
We divide $g$ by $p$ applying the Weierstrass Division Algorithm. Then there exist uniquely determined $q\in\mathbb C[[x,y]]$ and $r\in\mathbb C[[x]][y]$ such that $g(x,y)=q(x,y)p(x,y)+r(x,y)$, $\textrm{deg}_y(r)<m$ and $q$ is given by (\ref{QUOTIENT}).
By construction, $u,w\in\mathbb C[[x]][y]$, $\textrm{deg}_y(u)\leq n-m$ and $\textrm{deg}_y(w)\leq m-1$. Then $h(uw)\in\mathbb C[[x]][y]$ and has degree in $y$ at most $n-m+m-1-m=n-m-1$. Iterating this reasoning, we conclude that there is $\tau\in\mathbb N_0$ such that
\[
q=u+\sum_{\ell = 1}^{\tau}H^{\ell}(u).
\]
Note that by definition:
\begin{enumerate}
	\item[$(i)$] 
	If $p\in\mathbb C[[x]][y]$ then $w\in\mathbb C[[x]][y]$ and if $p\in\mathbb C[x][y]$ then $w\in\mathbb C[x][y]$;
	\item[$(ii)$] 
	If $g\in\mathbb C[[x]][y]$ then $u\in\mathbb C[[x]][y]$ and if $g\in\mathbb C[x][y]$ then $u\in\mathbb C[x][y]$.
\end{enumerate}
Then if $p,g\in\mathbb C[[x]][y]$ then $h(uw)\in\mathbb C[[x]][y]$ and if $p,g\in\mathbb C[x][y]$ then $h(uw)\in\mathbb C[x][y]$. Thus, for all $\ell\in\{1,\ldots,\tau\}$, if $p,g\in\mathbb C[[x]][y]$ then $H^{\ell}(u)\in\mathbb C[[x]][y]$ and if $p,g\in\mathbb C[x][y]$ then $H^{\ell}(u)\in\mathbb C[x][y]$.
\end{proof}

\noindent
For the following Lemma we assume statement $(A)$ of Theorem \ref{MAINTHEO} holds up to $i\in\{2,\ldots,s\}$ and  statement $(B)$ of Theorem \ref{MAINTHEO} holds up to $i-1$. We remind the reader that $f_0(x,y)=y$. For $j\in\{1,\ldots,s\}$, set $\Gamma_0(\lambda_1,\ldots,\lambda_j)$ as the subsemigroup of $\Gamma(\lambda_1,\ldots,\lambda_j)$ generated by $e_j$. 

\begin{lemma}\label{INDUI}
For all $i\in \{1,\ldots,s-1\}$ the following statements hold:
\begin{enumerate}
\item[$(A)$]
Let $\alpha,\beta\in\mathbb N_0$. Let $a(x,y),b(x,y)\in\mathbb C[[x,y]]$ non null such that $\vartheta_{\iota_i}(a),\vartheta_{\iota_i}(b)\in \Gamma_{i-1}(\lambda_1,\ldots,\lambda_i)$. If $\vartheta_{\iota_i}(a f_{i-1}^{\alpha})=\vartheta_{\iota_i}(b f_{i-1}^{\beta})$ then $\alpha-\beta$ is a multiple of $k_i$.
\item[$(B)$]
Let $g\in\mathbb C[[x]][y]$ non null such that $e_{i-1}\leq n<e_i$. Then there are uniquely determined $d\in\mathbb N$ and $a_{\ell}\in\mathbb C[[x]][y]$, $\ell\in\{0,\ldots,d\}$, such that $d<k_1$, 
\begin{equation}\label{DECOMPI}
g(x,y)=\sum_{\ell=0}^{d}a_{\ell}(x,y)f_{i-1}^{\ell}(x,y),
\end{equation}
$\vartheta_{\iota_i}(g)\in \Gamma(\lambda_1,\ldots,\lambda_i)$ and, for $a_{\ell}\neq 0$, $\textrm{deg}_y(a_{\ell})<e_{i-1}$. Furthermore, for all $j\in \{i,\ldots,s\}$ and $\sigma\in\mathbb N$,
\[
\vartheta_{\iota_j}(g)=\frac{e_j}{e_i}\vartheta_{\iota_i}(g) \textrm{ and } \vartheta_{\iota_j}(\partial^{\sigma}_y g)\geq \vartheta_{\iota_j}(g)-\sigma e_j\lambda_i.
\]

\end{enumerate}
\end{lemma}

\begin{proof}
We will prove the Lemma by induction on $i$. This Lemma for $i=1$ has been proved in Lemma \ref{INDUI1}. 

\noindent
We begin by proving statement $(A)$. We can assume $\alpha\geq \beta$. Equality $\vartheta_{\iota_i}(a f_{i-1}^{\alpha})=\vartheta_{\iota_i}(b f_{i-1}^{\beta})$ is equivalent to
\begin{equation}\label{VELEQ1}
(\alpha-\beta)\gamma^{(i)}_i=\vartheta_{\iota_i}(b)-\vartheta_{\iota_i}(a).
\end{equation}
Let $q$ and $r\in\{0,\ldots,k_i-1\}$ be the unique non negative integers such that
\begin{equation}\label{VELEQ1B}
\alpha-\beta=qk_i+r.
\end{equation}
We apply (\ref{VELEQ1B}) and (\ref{GEN}) to (\ref{VELEQ1}) and obtain
\begin{equation}\label{VELEQ2}
	e_i(r\lambda_i)=\vartheta_{\iota_i}(b)-\vartheta_{\iota_i}(a)-qe_i(k_i\lambda_i)-(\alpha-\beta)(k_{i-1}\gamma^{(i)}_{i-1}-e_i\lambda_{i-1}).
\end{equation}
By hypothesis, $\vartheta_{\iota_i}(b)-\vartheta_{\iota_i}(a)\in 
e_i\mathbb Z+\sum_{j=1}^{i-1}\mathbb Z \gamma^{(i)}_j$. Statement \em C3) \em (see Section \ref{SQO}) and statement \em (B) \em of Lemma \ref{ALTGENRETI} tell us that $qe_i(k_i\lambda_i)$ and $e_i\lambda_{i-1}$ belong to $e_i\mathbb Z+\sum_{j=1}^{i-1}\mathbb Z \gamma^{(i)}_j$. From statement $(C)$ of Proposition \ref{EQINEQ} we know that $k_{i-1}\gamma^{(i)}_{i-1}\in e_i\mathbb N_0+\sum_{j=1}^{i-1}\mathbb N_0 \gamma^{(i)}_j$. We conclude from (\ref{VELEQ2}) that $e_i(r\lambda_i)\in e_i\mathbb Z+\sum_{j=1}^{i-1}\mathbb Z \gamma^{(i)}_j$, which is equivalent, again by statement \em (B) \em of Lemma \ref{ALTGENRETI}, to $r\lambda_i\in M_{i-1}$. Since $r\in\{0,\ldots,k_i-1\}$, statement \em C3) of Section \ref{SQO} \em implies $r=0$.

\noindent
We now prove statement \em (B)\em. Let $\textrm{deg}_y(g)=n$ and $d,p$ be the unique non negative integers such that $n=d\,e_{i-1}+p$ and $p<e_{i-1}$. Since $e_{i-1}\leq n<e_i$ then $1\leq d<k_i$. The fact that $f_{i-1}$ is a Weierstrass polynomial with $\textrm{deg}_y(f_{i-1})=e_{i-1}$ and coefficients in $\mathbb C[x]$ implies that $f_{i-1}^{d}$ is a Weierstrass polynomial with $\textrm{deg}_y(f^d_{i-1})=de_{i-1}$ and coefficients in $\mathbb C[x]$. By Proposition \ref{WEIERQUOTIENT}, there exist uniquely determined $q$ and $r$ of $\mathbb C[[x]][y]$ such that $g=qf_{i-1}^d+r$, $\textrm{deg}_{y}(r)<d\,e_{i-1}$ and $\textrm{deg}_{y}(q)<n-d\,e_{i-1}=p$. If $\textrm{deg}_y(r)<e_{i-1}$ the result is proven. Otherwise we iterate the procedure dividing $r$ by an adequate power of $f_{i-1}$. Since the degree of the remainder of the Weierstrass Division strictly decreases, the procedure ends after a finite number of steps. Therefore, there are $d$ and $a_{\ell}\in\mathbb C[[x]][y]$, $\ell\in\{0,\ldots,d\}$, uniquely determined such that (\ref{DECOMPI}) holds and $\textrm{deg}_y(a_{\ell})<e_{i-1}$. By the induction hypothesis, for $a_{\ell}\neq 0$, there is $i_{\ell}\leq i-1$ such that $\vartheta_{\iota_{i_{\ell}}}(a_{\ell})\in \Gamma(\lambda_1,\ldots,\lambda_{i_{\ell}})$. Furthermore,
\[
\vartheta_{\iota_i}(a_{\ell})\in\frac{e_i}{e_{i_{\ell}}} \Gamma(\lambda_1,\ldots,\lambda_{i_{\ell}}).
\]
By Proposition \ref{SEMIREL},
\[
\vartheta_{\iota_i}(a_{\ell})\in \Gamma_{i_{\ell}}(\lambda_1,\ldots,\lambda_i).
\]
Since $i_{\ell}\leq i-1$ then $\Gamma_{i_{\ell}}(\lambda_1,\ldots,\lambda_i)$ is a subsemigroup of $\Gamma_{i-1}(\lambda_1,\ldots,\lambda_i)$ and we conclude that $\vartheta_{\iota_i}(a_{\ell})\in\Gamma_{i-1}(\lambda_1,\ldots,\lambda_i)$.

\noindent
Set
\begin{equation}\label{VALG}
\rho=\min\{\vartheta_{\iota_i}(a_{\ell}f_{i-1}^{\ell}):a_{\ell}\neq 0 \textrm{ and } \ell\in\{0,\ldots,d\}\}.
\end{equation}
Statement \em (A) \em implies that there is one and only one $\tau\in\{0,\ldots,n\}$ such that
\begin{equation}\label{ORDVALG}
\vartheta_{\iota_i}(a_{\tau}f_{i-1}^{\tau})=\rho. 
\end{equation}
Therefore $\vartheta_{\iota_i}(g)=\rho$. Since $\vartheta_{\iota_i}(a_{\tau})$ belongs to $\Gamma_{i-1}(\lambda_1,\ldots,\lambda_i)$, which is a subsemigroup of $\Gamma(\lambda_1,\ldots,\lambda_i)$, then
\[
\vartheta_{\iota_i}(g)=\vartheta_{\iota_i}(a_{\tau})+\tau\gamma_i^{(i)}\in \Gamma(\lambda_1,\ldots,\lambda_i).
\]
For all $j\in\{i,\ldots,s\}$ and $\ell\in\{0,\ldots,d\}$, the induction hypothesis implies that
\[
\vartheta_{\iota_j}(a_{\ell}f_{i-1}^{\ell})=\vartheta_{\iota_j}(a_{\ell})+\ell\vartheta_{\iota_j}(f_{i-1})=\frac{e_j}{e_{i_{\ell}}}\vartheta_{\iota_{i_\ell}}(a_{\ell})+\ell\gamma_i^{(j)}.
\]
Applying Proposition \ref{SEMIREL}, we obtain that
\[
\gamma_i^{(j)}=\frac{e_j}{e_i}\gamma_i^{(i)}.
\]
We once again apply  the induction hypothesis to conclude that
\[
\frac{e_j}{e_{i_{\ell}}}\vartheta_{\iota_{i_\ell}}(a_{\ell})=\frac{e_j}{e_i}\frac{e_i}{e_{i_{\ell}}}\vartheta_{\iota_{i_\ell}}(a_{\ell})=\frac{e_j}{e_i}\vartheta_{\iota_i}(a_{\ell}).
\]
Thus 
\[
\vartheta_{\iota_j}(a_{\ell}f_{i-1}^{\ell})=\frac{e_j}{e_i}\vartheta_{\iota_i}(a_{\ell})+\frac{e_j}{e_i}\ell\gamma_i^{(i)}=\frac{e_j}{e_i}\vartheta_{\iota_i}(a_{\ell}f_{i-1}^{\ell}).
\]
We conclude that
\[
\vartheta_{\iota_j}(g)=\frac{e_j}{e_i}\rho=\frac{e_j}{e_i}\vartheta_{\iota_i}(g).
\]
Before we proceed, it will be useful to compute, for $\ell\in \mathbb N$ and $j\in\{i,\ldots,s\}$, $\vartheta_{\iota_j}(\partial_y f_{i-1}^{\ell})$. By statement $(B)$ of Theorem \ref{MAINTHEO},
\begin{align}\label{VALDGPI}
	\vartheta_{\iota_j}(\partial_y f_{i-1}^{\ell})&=\vartheta_{\iota_j}( f_{i-1}^{\ell-1})+\vartheta_{\iota_j}(\partial_y f_{i-1})=(\ell-1)\gamma_i^{(j)}+\gamma^{(j)}_i-e_j\lambda_i=\\
	&=\vartheta_{\iota_j}(f_{i-1}^{\ell})-e_j\lambda_i.\nonumber
\end{align}
Let $\sigma\in\{1,\ldots,n\}$. We now prove the following inequality
\begin{equation}\label{HIGHORDDERI}
\vartheta_{\iota_j}(\partial^{\sigma}_y g)\geq \vartheta_{\iota_j}(g)-\sigma e_j\lambda_i
\end{equation}
by induction also on $\sigma$ (we remind the reader that we are assuming that the Lemma holds up to $i-1$). Note that (\ref{HIGHORDDERI}) holds trivially for $\ell\geq n+1$. 

\noindent
We start by proving inequality (\ref{HIGHORDDERI}) with $\sigma=1$. For $\ell\in\{0,\ldots,d\}$, since $i_{\ell}\leq i-1$ and \em C1) \em holds then $\lambda_{i_{\ell}}\leq \lambda_{i-1}<\lambda_i$. By the induction hypothesis on $i$,
\begin{equation}\label{VALDG0I}
\vartheta_{\iota_j}(\partial_y a_0)\geq \vartheta_{\iota_j}(a_0)-e_j\lambda_{i_0}> \frac{e_j}{e_i}\vartheta_{\iota_i}(a_0)-e_j\lambda_{i}\geq \frac{e_j}{e_i}\rho-e_j\lambda_{i}.
\end{equation}
For $\ell\in\{1,\ldots,d\}$, from the inequality
\begin{eqnarray*}
\vartheta_{\iota_j}(f_{i-1}^{\ell}\partial_y a_{\ell})-\vartheta_{\iota_j}(a_{\ell}\partial_y f_{i-1}^{\ell})\geq\hspace{6cm}\\ 
\geq\left(\vartheta_{\iota_j}(f_{i-1}^{\ell})+\vartheta_{\iota_j}(a_{\ell})-e_j\lambda_{i_{\ell}}\right)-\left(\vartheta_{\iota_j}(a_{\ell})+\vartheta_{\iota_j}(f_{i-1}^{\ell})-e_j\lambda_i\right)=\hspace{1.1cm}\\
=e_j(\lambda_i-\lambda_{i_{\ell}})>0\hspace{8.2cm}
\end{eqnarray*}
we obtain that
\begin{equation}\label{VALDGI}
\vartheta_{\iota_j}\left(\partial_y (a_{\ell}f_{i-1}^{\ell})\right)=\vartheta_{\iota_j}(a_{\ell}\partial_y f_{i-1}^{\ell})=\frac{e_j}{e_i}\vartheta_{\iota_i}(a_{\ell}f_{i-1}^{\ell})-e_j\lambda_i\geq \frac{e_j}{e_i}\rho-e_j\lambda_{i}.
\end{equation}
We conclude from inequalities (\ref{VALDG0I}) and (\ref{VALDGI}) that
\[
\vartheta_{\iota_j}(\partial_y g)\geq \frac{e_j}{e_i}\rho-e_j\lambda_{i}=\vartheta_{\iota_j}(g)-e_j\lambda_{i}.
\]
Assume that $\sigma\in\{2,\ldots,n-e_{i-1}+1\}$. Then $\textrm{deg}_y(\partial^{\ell-1}g)\geq e_{i-1}$. By the induction hypothesis on $\sigma$,
\[
\vartheta_{\iota_j}\left(\partial_y (\partial^{\sigma-1}_y g)\right)\geq \vartheta_{\iota_j}(\partial^{\sigma-1}_y g)-e_j\lambda_i\geq \vartheta_{\iota_j}(g)-\sigma e_j\lambda_i.
\]
Assume now that $\sigma\in\{n-e_{i-1}+2,\ldots,n\}$. Set $\varsigma=n-e_{i-1}+1$. Then  $\textrm{deg}_y(\partial^{\varsigma}g)< e_{i-1}$. By the induction hypothesis on $i$,
\[
\vartheta_{\iota_j}\left(\partial^{\sigma-\varsigma}_y (\partial^{\varsigma}_y g)\right)\geq \vartheta_{\iota_j}(\partial^{\varsigma}_y g)-(\sigma-\varsigma)e_j\lambda_{i-1}> \vartheta_{\iota_j}(\partial^{\varsigma}_y g)-(\sigma-\varsigma)e_j\lambda_{i}.
\]
On the other hand, by the induction hypothesis on $\sigma$,
\[
\vartheta_{\iota_j}(\partial^{\varsigma}_yg)\geq \vartheta_{\iota_j}(g)-\varsigma e_j\lambda_i.
\]
Thus
\[
\vartheta_{\iota_j}\left(\partial^{\sigma-\varsigma}_y (\partial^{\varsigma}_y g)\right)\geq\vartheta_{\iota_j}(g)-\sigma e_j\lambda_i
\]
\end{proof}

\noindent
We would like to call the attention of the reader to the following facts:
\begin{itemize}
	\item[$R1)$]
	In statement \em (B) \em of Lemma \ref{INDUI}, if we consider $g\in\mathbb C[x][y]$ then the $a_{\ell}$'s in (\ref{DECOMPI}) are elements of $\mathbb C[x,y]$. In the proof of this statement, since $f_{i-1}$ is also an element of  $\mathbb C[x][y]$, Proposition \ref{WEIERQUOTIENT} assures us that $q\in\mathbb C[x][y]$ and as a consequence $r\in\mathbb C[x][y]$.
	\item[$R2)$]
	Once we have determined $f_1,\ldots,f_{i-1}$ the decomposition (\ref{DECOMPI}) allows us, by induction, to write $g$ uniquely as a linear combination of the type
	\[
	\sum_{\ell=1}^{\widetilde{d}}c_{\ell} x^{\alpha_\ell}y^{\beta_{0,\ell}}(x,y)f_1^{\beta_{1,\ell}}(x,y)\cdots f_{i-1}^{\beta_{i-1,\ell}}(x,y),
	\]
	where $\widetilde{d}\in\mathbb N$, $c_{\ell}\in\mathbb C^{\ast}$ and $(\alpha_{\ell},\beta_{0,\ell},\ldots,\beta_{i-1,\ell})\in\mathbb N_0^{i+1}\setminus\{(0,\ldots,0)\}$, for all $\ell\in\{1,\ldots,\widetilde{d}\}$.
\end{itemize}

\begin{proof}[Proof of statement $(B)$ of the Main Theorem]
Lemma \ref{MAINTHEOI1} proves statement $(B)$ of Theorem \ref{MAINTHEO} for $i=1$. Let $i\in\{2,\ldots,s-1\}$. Assume that statement $(A)$ of Theorem \ref{MAINTHEO} holds up to $i$ and  statement $(B)$ of Theorem \ref{MAINTHEO} holds up to $i-1$. Then Lemma \ref{INDUI} holds up to $i-1$. Some computations made in the proof of Lemmas \ref{INDUI1} and \ref{INDUI} are also useful for this proof, so we will make some references to them. 

\noindent
Let $\delta_i$ be as in statement \em (A) \em of Theorem \ref{MAINTHEO} and $n=\textrm{deg}_y(\delta_i)$. A similar reasoning to the one used to prove equality (\ref{VALf1}) allows us to conclude that
\[
\vartheta_{\iota_i}(\delta_i)=k_i\gamma_i^{(i)}=\vartheta_{\iota_i}(f_{i-1}^{k_i}).
\]
Since $\textrm{Supp}(\delta_i)\subset N_i\setminus\{(0,e_i)\}$ then $n<e_i$ and we can write 
\[
\delta_i(x,y)=\sum_{\ell=0}^{d}a_{\ell}(x,y)f_{i-1}^{\ell}(x,y)
\]
as in statement $(B)$ of Lemma \ref{INDUI}. Note that if $n<e_{i-1}$ then $d=0$. Let $\rho$ be as in (\ref{VALG}). Given that $\tau\in\{0,\ldots,d\}$ and $d<k_i$,  statement \em (A) \em of Lemma \ref{INDUI} implies that there is one and only one $\tau$ such that  (\ref{ORDVALG}) holds. Furthermore, $\vartheta_{\iota_i}(\delta_i)=\rho$ which implies
\[ 
\vartheta_{\iota_i}(a_{\tau}f_{i-1}^{\tau})=\vartheta_{\iota_i}(f^{k_i}).
\]
Hence, once again by statement \em (A) \em of Lemma \ref{INDUI}, we conclude that $\tau=0$. Let $j\in\{i+1,\ldots,s\}$ and set
\[
\epsilon_{i,j}=c_{i}t^{e_j\lambda_{i+1}}+\varphi_{i+1}(t^{e_j/k})+\cdots+c_jt^{e_j\lambda_j}+\varphi_j(t^{e_j/k}).
\]
Therefore
\[
\iota_{j}^{\ast}x=(\iota_i\circ \chi_{\frac{e_j}{e_i}})^{\ast}x \textrm{ and } \iota_{j}^{\ast}y=(\iota_i\circ \chi_{\frac{e_j}{e_i}})^{\ast}y+\epsilon_{i,j}.
\]
By Taylor's formula,
\begin{align}\label{TAYLORFI}
	\iota_j^{\ast}f_i=&(\iota_i\circ \chi_{\frac{e_j}{e_i}})^{\ast}f_i(x,y)+\epsilon_{i,j} (\iota_i\circ \chi_{\frac{e_j}{e_i}})^{\ast}\partial_y f_{i-1}^{k_i}+\\	&+\sum_{\ell=2}^{e_i}\frac{1}{\ell!}\epsilon_{i,j}^{\ell}(\iota_i\circ \chi_{\frac{e_j}{e_i}})^{\ast}\partial_y^{\ell} f_{i-1}^{k_i}+
	\sum_{\ell=1}^{n}\frac{1}{\ell!}\epsilon_{i,j}^{\ell}(\iota_i\circ \chi_{\frac{e_j}{e_1}})^{\ast}\partial^{\ell}_y \delta_i(x,y).\nonumber
\end{align}
By construction $\iota_i^{\ast}f_i=0$ which implies $(\iota_i\circ \chi_{\frac{e_j}{e_i}})^{\ast}f_i=0$. Let us prove that
\begin{equation}\label{ORDJFI}
\textrm{ord}(\iota_j^{\ast}f_i)=\textrm{ord}\left(\epsilon_{i,j} (\iota_i\circ \chi_{\frac{e_j}{e_i}})^{\ast}\partial_y f_{i-1}^{k_i}\right).
\end{equation}
To do so, it suffices to prove that, for all $\ell\in\{1,\ldots,n\}$,
\begin{equation}\label{INEQ1VALFI}
	\textrm{ord}\left(\epsilon_{i,j}^{\ell}(\iota_i\circ \chi_{\frac{e_j}{e_1}})^{\ast}\partial^{\ell}_y \delta_i\right)>\textrm{ord}\left(\epsilon_{i,j} (\iota_i\circ \chi_{\frac{e_j}{e_i}})^{\ast}\partial_y f_{i-1}^{k_i}\right)
\end{equation}
and that, for all $\ell\in\{2,\ldots,e_i\}$,
\begin{equation}\label{INEQ2VALFI}
\textrm{ord}\left(\epsilon_{i,j}^{\ell}(\iota_i\circ \chi_{\frac{e_j}{e_i}})^{\ast}\partial_y^{\ell} f_{i-1}^{k_i}\right)>\textrm{ord}\left(\epsilon_{i,j} (\iota_i\circ \chi_{\frac{e_j}{e_i}})^{\ast}\partial_y f_{i-1}^{k_i}\right)
\end{equation}
We start by proving (\ref{INEQ1VALFI}). Set $\overline{\delta}_i=\delta_i(x,y)-a_0(x,y)$. Then either $\overline{\delta}_i$ is null or verifies the conditions of statement $(B)$ of Lemma \ref{INDUI}. Furthermore $\textrm{deg}_y(a_0)<e_{i-1}$. Let $\ell\in\mathbb N$ and $i_0$ be the smallest non negative integer such that $\textrm{deg}_y(a_0)<e_{i_0}$. Then $i_0\leq i-1$ and, once again by statement $(B)$ of Lemma \ref{INDUI}, $\vartheta_{\iota_{\sigma}}\left(\partial_y^{\ell}a_0\right)\geq \vartheta_{\iota_{\sigma}}\left(a_0\right)-\ell e_{\sigma}\lambda_{i_0}$ for all $\sigma\in\{i,\ldots,s\}$. Since $\lambda_{i_0}\leq \lambda_{i-1}$ and $\tau=0$ then $\vartheta_{\iota_i}\left(a_0\right)-\ell e_i\lambda_{i_0}\geq k_i\gamma_i^{(i)}-\ell e_i\lambda_{i-1}$. Thus
\begin{equation}\label{DERA0J}
\vartheta_{\iota_i}\left(\partial_y^{\ell}a_0\right)\geq k_i\gamma_i^{(i)}-\ell e_i\lambda_{i-1},
\end{equation}
From (\ref{DERA0J}) and (\ref{VALDGPI}), we conclude that
\begin{align}\label{INEQVALDELTA1}
	\textrm{ord}\left(\epsilon_{i,j}^{\ell}(\iota_i\circ \chi_{\frac{e_j}{e_i}})^{\ast}\partial_y^{\ell} a_0\right)-\textrm{ord}\left(\epsilon_{i,j} (\iota_i\circ \chi_{\frac{e_j}{e_i}})^{\ast}\partial_y f_{i-1}^{k_i}\right)\geq \hspace{2.5cm}&\\
	\nonumber
	\geq \left(\ell e_j\lambda_{i+1}+\frac{e_j}{e_i}\left( k_i\gamma_i^{(i)}-\ell e_i\lambda_{i-1}\right)\right)- \left(e_j\lambda_{i+1}+\frac{e_j}{e_i}\left(k_i\gamma_i^{(i)}-e_i\lambda_i\right)\right)=&\\
	\nonumber
	=e_j\left((\ell-1)(\lambda_{i+1}-\lambda_{i-1})+\lambda_i-\lambda_{i-1}\right)>0.\hspace{4.5cm}&
\end{align}
If $\overline{\delta}_i$ is non null then, for all $\sigma\in\{i,\ldots,s\}$,
\begin{equation}\label{DEROVERDELTAJ}
\vartheta_{\iota_{\sigma}}\left(\partial_y^{\ell}\overline{\delta}_i\right)\geq \vartheta_{\iota_{\sigma}}\left(\overline{\delta}_i\right)-\ell e_{\sigma}\lambda_i.
\end{equation}
Furthermore, since $\tau=0$, $\vartheta_{\iota_i}\left(\overline{\delta}_i\right)>\vartheta_{\iota_i}\left(a_0\right)$. Therefore 
\begin{equation}\label{DEROVERDELTAI}
\vartheta_{\iota_i}\left(\partial_y^{\ell}\overline{\delta}_i\right)>  k_i\gamma_i^{(i)}-\ell e_i\lambda_i. 
\end{equation}
Hence
\begin{align}\label{INEQVALDELTA2}
	\textrm{ord}\left(\epsilon_{i,j}^{\ell}(\iota_i\circ \chi_{\frac{e_j}{e_i}})^{\ast}\partial_y^{\ell}\overline{\delta}_i \right)-\textrm{ord}\left(\epsilon_{i,j} (\iota_i\circ \chi_{\frac{e_j}{e_i}})^{\ast}\partial_y f_{i-1}^{k_i}\right)> \hspace{2.5cm}&\\
	\nonumber
	>\left(\ell e_j\lambda_{i+1}+\frac{e_j}{e_i}\left(k_i\gamma_i^{(i)}-\ell e_i\lambda_i\right)\right)- \left(e_j\lambda_{i+1}+\frac{e_j}{e_i}\left(k_i\gamma_i^{(i)}-e_i\lambda_i\right)\right)=\hspace{0.3cm}&\\
	\nonumber
	=e_j(\ell-1)(\lambda_{i+1}-\lambda_i)\geq 0.\hspace{7cm}&
\end{align}
Inequalities (\ref{INEQVALDELTA1}) and (\ref{INEQVALDELTA2}) imply that (\ref{INEQ1VALFI}) holds.

\noindent
Let $\ell\geq 2$. To prove (\ref{INEQ2VALFI}), note that $\textrm{deg}_y(\partial_y^{\ell} f_{i-1}^{k_i})<e_i$ and that $\partial_y^{\ell} f_{i-1}^{k_i}$ verifies the conditions of 
statement $(B)$ of Lemma \ref{INDUI}. Using similar arguments to the ones used in the proof of inequality (\ref{DEROVERDELTAI}), we prove that
\[
\vartheta_{\iota_i}\left(\partial_y^{\ell} f_{i-1}^{k_i}\right)\geq k_i\gamma_i^{(i)}-\ell e_i\lambda_i.
\]
and consequently
\begin{align*}
\textrm{ord}\left(\epsilon_{i,j}^{\ell}(\iota_i\circ \chi_{\frac{e_j}{e_i}})^{\ast}\partial_y^{\ell} f_{i-1}^{k_i}\right)-\textrm{ord}\left(\epsilon_{i,j} (\iota_i\circ \chi_{\frac{e_j}{e_i}})^{\ast}\partial_y f_{i-1}^{k_i}\right)\geq&\\
\geq e_j(\ell-1)(\lambda_{i+1}-\lambda_i)> 0.\hspace{5cm}&
\end{align*}
Equality (\ref{ORDJFI}) allow us to conclude that
\begin{align*}
	\vartheta_{\iota_j}(f_i)=&\textrm{ord}(\iota_j^{\ast}f_i)=\textrm{ord}\left(\epsilon_{i,j} (\iota_i\circ \chi_{\frac{e_j}{e_i}})^{\ast}\partial_y f_{i-1}^{k_i}\right)=\\
	=&e_j\lambda_{i+1}+\frac{e_j}{e_i}k_i\gamma_i^{(i)}-e_j\lambda_i=
	\frac{e_j}{e_i}\left(e_i\lambda_{i+1}+k_i\gamma_i^{(i)}-e_i\lambda_i\right)=\frac{e_j}{e_i}\gamma_{i+1}^{(i)}.
\end{align*}
By Proposition \ref{SEMIREL}, $\frac{e_j}{e_i}\gamma_{i+1}^{(i)}=\gamma_{i+1}^{(j)}$. Hence
$\vartheta_{\iota_j}(f_i)=\gamma_i^{(j)}$.

\noindent
We now prove that $\vartheta_{\iota_j}(\partial_y f_i)=\gamma_{i+1}^{(j)}-e_j\lambda_{j+1}$. Let $\ell\in\{1,\ldots,d\}$. If $\overline{\delta}_i$ is non null then, taking into account (\ref{DEROVERDELTAJ}) and statement $(B)$ of Lemma \ref{INDUI},
\begin{equation}\label{INEQ1VALFJ0}
\vartheta_{\iota_j}\left(\partial_y^{\ell}\overline{\delta}_i\right)>\frac{e_j}{e_i}k_i\gamma_i^{(i)}-e_j\lambda_i.
\end{equation}
Combining inequality (\ref{INEQ1VALFJ0}) with (\ref{VALDGI}) one obtains that 
\begin{equation}\label{INEQ1VALFJ}
\vartheta_{\iota_j}\left(\partial_y^{\ell}\overline{\delta}_i\right)-\vartheta_{\iota_j}(\partial_y f_{i-1}^{k_i})>0.
\end{equation}
From (\ref{DERA0J}) we conclude that
\begin{align}\label{INEQ2VALFJ}
\vartheta_{\iota_j}\left(\partial_y a_0\right)-\vartheta_{\iota_j}(\partial_y f_{i-1}^{k_i})\geq & \frac{e_j}{e_i}k_i\gamma_i^{(i)}-e_j\lambda_{i-1}-\left(\frac{e_j}{e_i}k_i\gamma_i^{(i)}-e_j\lambda_i\right)=\\
\nonumber
=& e_j(\lambda_i-\lambda_{i-1})>0.
\end{align}
Inequalities (\ref{INEQ1VALFJ}) and (\ref{INEQ2VALFJ}) imply that $\vartheta_{\iota_j}(\partial_y f_i)=\vartheta_{\iota_j}(\partial_y f_{i-1}^{k_i})$. Therefore
\[
\vartheta_{\iota_j}(\partial_y f_i)=\frac{e_j}{e_i}k_i\gamma_i^{(i)}-e_j\lambda_i=k_i\gamma_i^{(j)}-e_j\lambda_i.
\]
But, by (\ref{GEN}), $k_i\gamma_i^{(j)}-e_j\lambda_i=\gamma_{i+1}^{(j)}-e_j\lambda_{j+1}$ and the result is proved.
\end{proof}

\section{Computational Application}\label{CA}

\noindent
In this Section we present an algorithm, Algorithm \ref{COMPIMP}, based on the proof of Theorem \ref{MAINTHEO}, to compute $f_i$. We also present two examples where we compare computing times between an implementation of this algorithm and elimination theory using Gr\"{o}bner basis.

\noindent
As input for this algorithm we need the polynomials $f_0,\ldots,f_{i-1}$, the list $CE_{i-1}$ of the characteristic exponents $\lambda_1,\ldots,\lambda_{i-1}$, the list $SG_{i-1}$ of the generators of the semigroup associated to characteristic exponents $CE_{i-1}$, the parametrization $\iota_{i-1}$ given by $x_{i-1}$ and $y_{i-1}$, the characteristc exponent $\lambda_i$, the positive integer $k_i$ and the polynomial $\varphi_i$. We now explain the algorithm

\noindent
\em Lines 1 to 6: \em We update the lists $CE_{i-1}$ and $SG_{i-1}$. We compute $\iota_i$ and store it in $x_i$ and $y_i$.

\noindent
\em Lines 7 to 11: \em The list $LS_i$ stores the degree of the polynomials $\iota_i^{\ast}x, \iota_i^{\ast}f_0,\ldots,\iota_i^{\ast}f_{i-1}$. These degrees will play an important role in the elimination procedure. The lists $VE$ and $VC$ contain variables needed for future computations.

\noindent
\em Lines 12 to 14: \em We initiate step 2 of the proof of Theorem \ref{MAINTHEO}. For control, $j$ stores the number of iterations made so far and $n$ the valuation we are currently applying to the elimination procedure. For information on how the algorithm is running, we decided to display the current valuation being eliminated.

\noindent
\em Lines 15 to 16: \em We start a loop that will only end when we obtain $g$ such that $\vartheta_{\iota_i}(g)=+\infty$.  For line $15$ we assume that we have access to a procedure, denoted by \textbf{intregion}, that computes all $(i+1)$-uples with entries positive integers, on a compact region of a hyperplane and we store those uples on the list $E$. The equality 
\[
VE\cdot SG_i==n
\] 
defines the hyperplane accordingly to the current valuation we are eliminating. The support of $g$ must be contained in the set $N_i$. This fact is assured by the inequality 
\[
VE\cdot LS_i\leq SG_i(1)*\textrm{\bf{Deg}}(y_{i}(t))
\]
which comes from characterization (\ref{NiALT}) of $N_i$. Hence this elimination algorithm is finite. The main computational complexity of this algorithm arises from the integer linear programming problem of obtaining the list $E$.

\noindent
\em Lines 17 to 24: \em In this loop we compute the decomposition stated in remark $R2)$ from the elements of the list $E$. We also update the list $VC$ with the coefficients of the decomposition introduced in each cycle of the loop. In the first iteration of the \textbf{While} loop, $(0,\ldots,k_i)$ is an element of $E$ that we must not use.

\noindent
\em Lines 25 to 30: \em We solve the linear homogeneous equation, with variables the elements of $VC$, obtained by requiring that the coefficient of $t^n$ in $\iota_{i}^{\ast}g$ is zero. For computational reason we are assuming that solution of this linear homogeneous equation is given in rule form so we apply it to $g$.

\begin{algorithm}\label{COMPIMP}
	\SetKwInOut{Input}{input}\SetKwInOut{Output}{output}
	\SetKwFunction{Length}{Length}
	\SetKwFunction{Deg}{Deg}
	\SetKwFunction{Ord}{Ord}
	\SetKwFunction{IntRegion}{IntRegion}
	\SetKwFunction{Display}{Display}
	\SetKwFunction{LinSolve}{LinSolve}
	\SetKwFunction{Coef}{Coef}
	\SetKwFunction{ApplyRule}{ApplyRule}
	\Input{$k_i$, $\lambda_i$, $\varphi_i$, $x_{i-1}$, $y_{i-1}$, $SG_{i-1}$, $CE_{i-1}$, $f_0,\ldots,f_{i-1}$}
	\Output{$f_i$}
	
	\BlankLine
	List $CE_i=CE_{i-1}\cup\{\lambda_i\}$\;
	List $SG_i=k_i\cdot SG_{i-1}$\;
	Int $\gamma_i=k_i* SG_{i-1}(i)-SG_i(1)* CE_{i-1}(i-1)+SG_i(1)* \lambda_i$\;
	List $SG_i=SG_{i-1}\cup\{\gamma_i\}$\;
	Poly $x_i(t)=x_{i-1}(t^{k_i})$\;
	Poly $y_i(t)=y_{i-1}(t^{k_i})+c_it^{SG_i(1)*\lambda_i}+\varphi_{i}(t)$\;
	List $LS_i=\{SG_i(1),$\Deg{$y_{i}(t)$}$\}$\;
	\For{$\ell=1$ \KwTo $i-1$}{
		List $LS_i=LS_i\cup\{$\Deg{$f_{\ell}(x_i(t),y_i(t))$}$\}$
	}
	List $VE=\{\alpha,\beta_0,\ldots,\beta_{i-1}\}$; List $VC=\{\}$\;
	Poly $g(x,y)=f_{i-1}^{k_i}(x,y)$\;
	Int $j=1$; Int $n=$\Ord{$g(x_i(t),y_i(t))$}\; 
	\Display{n}\;
	\While{$n\neq +\infty$}{
		List E=\IntRegion{$VE\cdot SG_i==n\, \wedge\, VE\cdot LS_i\leq SG_i(1)*$\Deg{$y_{i}(t)$}$\,\wedge\, VE\geq \{0,\ldots,0\}$}\;
		\For{$\ell=1$ \KwTo \Length{$E$}}{
			\eIf{$j\neq 1$}
			{Poly $g(x,y)=g(x,y)+e_{j,\ell}x^{E(\ell)(1)}\prod_{m=2}^{i+1} f_{m-2}^{E(\ell)(m)}$}
			{\If{$E(\ell)\neq \{0,\ldots,0,k_i\}$}
				{Poly $g(x,y)=g(x,y)+e_{j,\ell}x^{E(\ell)(1)}\prod_{m=2}^{i+1} f_{m-2}^{E(\ell)(m)}$}
			}
			List $VC=VC\cup \{e_{j,\ell}\}$\;
		}
		Rule S=\LinSolve{\Coef{$g(x_i(t),y_i(t)),n$}==0,$VC$}\;
		Poly g(x,y)=\ApplyRule{S,$g(x,y)$}\;
		Int $n=$\Ord{$g(x_i(t),y_i(t))$}\; 
		Int $j=j+1$\;
		\Display{n}\;
	}
	\caption{Elimination Procedure}
	Poly $f_i(x,y)=g(x,y)$\;
\end{algorithm}

\noindent
For the following two examples, we implemented Algorithm \ref{COMPIMP} in \em Mathematica \em and compared the computational times with a elimination using Gr\"{o}bner basis, with a global order, in \em Singular\em.

\begin{example}\label{EX2.2}
Let $\zeta$ be as in Example \ref{EX2.1}. Assume that $f_1$ is known. In the \em Mathematica \em implementation, it took roughly 3 seconds to compute $f_2$. In \em Singular\em, after one hour, it still hadn't computed $f_2$ so the authors decided to terminate the computation. The authors decided to repeat the \em Singular \em session, now  giving values to the coefficients of $\zeta$, and obtained $f_2$ in roughly one second.
\end{example}

\begin{example}\label{EX3}
Let $c_{\ell}\in\mathbb C^{\ast}$, $i=1,2,3$ and $\zeta=c_1 x^{\frac{6}{5}}+c_2 x^{\frac{3}{2}}+c_3 x^{\frac{5}{3}}$. This branch has three characteristic exponents. Assume that $f_1$ and $f_2$ are known. In the \em Mathematica \em implementation, it took roughly 6 minutes and 33 seconds to compute $f_3$. In \em Singular\em, after one hour, it still hadn't computed $f_3$ (even after assigning values to the coefficients) so the authors decided to terminate the computation.
\end{example}


\begin{thebibliography}{9}
	
	\bibitem{BK} Brieskorn, E., Kn\"{o}rrer, H.: Plane Algebraic Curves. Birkh\"{a}user Verlag, Switzerland (1986).
	
	\bibitem{JO} Cabral, J., Neto, O.: Microlocal versal deformations of the plane curves $y^k = x^n$. C. R. Acad. Sci. Paris, Ser. I 347, 1409–1414 (2009).
	
	\bibitem{SING} Decker, W., Greuel, G. -M., Pfister, G., Sch{\"o}nemann, H.: Singular 4-2-1 --- A computer algebra system for polynomial computations. \url{http://www.singular.uni-kl.de} (2022)
	
	\bibitem{CM} D’Andrea, C., Sombra, M.: The Newton Polygon of a Rational Plane Curve. Math.Comput.Sci. 4,3–24 (2010).
	
	\bibitem{GP1} González Pérez, P.: The semigroup of a quasi-ordinary hypersurface. J. Inst. Math. Jussieu 2(3), 383–399 (2003).
	
	\bibitem{GLS} Greuel, G. -M., Lossen, C., Shustin, E.: Introduction to Singularities and Deformations. Springer, Heidelberg (2007).
	
	\bibitem{GP2} Greuel, G. -M., Pfister, G.: A Singular Introduction to Commutative Algebra. Springer, Heidelberg (2008).
	
	\bibitem{JP}  Jong, T., Pfister, G.: Local Analytic Geometry. Vieweg, Braunschweig/Wiesbaden (2000).
	
	\bibitem{LIP} Lipman, J.: Quasi-ordinary singularities of embedded surfaces. Thesis (Ph. D.)-Harvard University, Massachusetts (1965).
	
	\bibitem{MSTY} Maclagan, D., Sturmfels, B.: Introduction to tropical geometry. Graduate Studies in Mathematics 161, American Mathematical Society, Rhode Island (2015)
	
	\bibitem{STY}  Sturmfels, B., Tevelev, J., Yu, J.: The Newton polytope of the implicit equation. Moscow Math. J. 7, 327–346 (2007).
	
	\bibitem{WALL} Wall, C.T.C.: Singular points of plane curves. London. Math. Soc., Students Texts 63,  Cambridge University Press, New York (2004).
	
	\bibitem{MATH} Wolfram Research, Inc.: Mathematica, Version 13.2. \url{https://www.wolfram.com/mathematica}, Champaign, IL (2022).
	
	\bibitem{ZAR} Zariski, O.: The Moduli Problem for Plane Branches. University Lecture Series 39, American Mathematical Society, Rhode Island (2006).
	
\end{thebibliography}
\end{document}